\documentclass[letterpaper,12pt]{article}
\usepackage{latexsym}
\usepackage{graphicx,graphics}
\usepackage{graphics,color,colortbl}
\usepackage{enumerate}
\usepackage{tikz}
\usepackage{amssymb,amsmath,amsthm}
\usepackage[dvips]{epsfig}
\usepackage[activeacute, english]{babel}
\usepackage[ansinew]{inputenc}
\usepackage{hyperref}
\usepackage{tocloft}
\usepackage{titlesec}
\usepackage{float}
\usetikzlibrary{arrows,decorations.markings}
\tikzstyle{dashed}=[dash pattern=on 9pt off 3pt]
\usepackage[figurename=Figure]{caption}

\newtheorem{theorem}{Theorem}[section]
\newtheorem*{conjecture*}{Conjecture}
\newtheorem{conjecture}{Conjecture}[section]

\newtheorem{lemma}[theorem]{Lemma}
\newtheorem{proposition}[theorem]{Proposition}

\theoremstyle{definition}
\newtheorem{definition}[theorem]{Definition}

\newtheorem{example}[theorem]{Example}

\theoremstyle{remark}
\newtheorem{remark}[theorem]{Remark}

\newtheorem{notation}[theorem]{Notation}

\titleformat*{\section}{\large\bfseries}
\titleformat*{\subsection}{\normalsize}

\begin{document}
	
	\title{On the Forking Path Conjecture}
	\author{Gonzalo Jim\'enez}
	\date{\today}
	\maketitle
	
	\begin{abstract}
		\noindent 	
		We prove the Forking Path Conjecture for all but one element in the symmetric group $S_4$. Two specific paths in the rex graph of that element give a counterexample for the conjecture. We propose a refined conjecture for the longest element of any $S_n$.
	\end{abstract}
	\vspace{-1cm}
	\renewcommand\contentsname{\small \begin{center}
			Contents\vspace{-0.5cm}
	\end{center}}
	\setcounter{tocdepth}{1}
	\setlength{\cftbeforesecskip}{0cm}
	\setlength{\cftparskip}{0cm}
	\vspace{-0.1cm}
	\tableofcontents
	\section{Introduction}
	
	In the 2011 paper \cite{Libedinsky2011new}, N. Libedinsky studied morphisms induced by paths in the reduced expression graph (see Definition \ref{defrexgraph}) of extra-large Coxeter systems. He showed the surprising fact that morphisms induced by complete paths are idempotents on the corresponding Bott-Samelson bimodules. 
	In 2016, B. Elias provided in \cite{elias2016thicker} an extension of his work with M. Khovanov \cite{elias2010diagrammatics}, where they gave a diagrammatic presentation of the category of Bott-Samelson bimodules $\mathbb{BS}$Bim. Here, morphisms can be translated into linear combinations of planar graphs, and stacking planar graphs can be interpreted as composing morphisms. Elias used their diagrammatic calculus to construct an idempotent using the reduced expression graph of the longest element $w_0$, in the symmetric group $S_n$. 
	These results, plus thousands of cases checked by computer, motivated Libedinsky in 2017 to announce the {\it Forking Path Conjecture} \cite[Section 6.3]{libGentle}.
	\begin{conjecture*}[Forking Path Conjecture]\label{fpc}
		Let $x \in S_n$, let $p, q$ be two complete paths with the same starting points and the same ending points in the reduced expression graph of $x$. The morphisms induced by these paths are equal.
	\end{conjecture*}
	
	In this document we prove the conjecture for all but one element in $S_4$. The outstanding element is the one that sends 1 to 4, 2 to 2, 3 to 3, and 4 to 1. Viewed in the context of Coxeter systems with generators $s_i=(i\ i+1)$, its reduced expression graph is the following (we simplify notation writing $ijk$ in place of $s_is_js_k$). 
	
	\begin{figure}[H]
		\centering
		\tikzset{every picture/.style={line width=0.75pt}} 
		\begin{tikzpicture}[x=0.45pt,y=0.45pt,yscale=-1,xscale=1]
			
			\draw [line width=1.5]  
			(160,40) -- (190,10) ;
			\draw [line width=1.5]    (60,50) -- (95,50) ;
			\draw [line width=1.5]  
			(260,10) -- (290,40) ;
			\draw [line width=1.5]    (360,50) -- (395,50) ;
			\draw [line width=1.5]  
			(260,90) -- (290,60) ;
			\draw [line width=1.5]  
			(160,60) -- (190,90) ;
			
			\draw (0,40) node [anchor=north west][inner sep=0.75pt]  [font=\scriptsize]  {$12321$};
			\draw (200,0) node [anchor=north west][inner sep=0.75pt]  [font=\scriptsize]  {$13213$};
			\draw (300,40) node [anchor=north west][inner sep=0.75pt]  [font=\scriptsize]  {$31213$};
			\draw (100,40) node [anchor=north west][inner sep=0.75pt]  [font=\scriptsize]  {$13231$};
			\draw (400,40) node [anchor=north west][inner sep=0.75pt]  [font=\scriptsize]  {$32123$};
			\draw (200,85) node [anchor=north west][inner sep=0.75pt]  [font=\scriptsize]  {$31231$};
			
		\end{tikzpicture}
		\caption{Reduced expression graph of $12321$} 
	\end{figure}
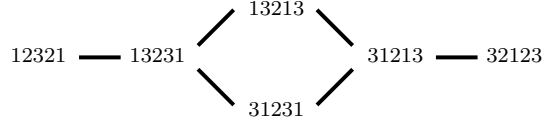
	
	The counterexample is given by $1 \otimes_{s_1} 1 \otimes_{s_3} 1 \otimes_{s_2}  x_3 \otimes_{s_3} 1 \otimes_{s_1} 1$, in the Bott-Samelson bimodule $B_{s_1}B_{s_3}B_{s_2}B_{s_3}B_{s_1}$. The complete paths inducing the different morphisms are the following.
	\begin{figure}[H]
		\centering
		\footnotesize{ $v_1:=13231 \rightarrow  31231 \rightarrow  31213 \rightarrow  32123 \rightarrow  31213 \rightarrow  13213 \rightarrow  13231 \rightarrow  12321 \rightarrow  13231$.}\\ 
		\vspace{0.2cm}
		\footnotesize{ $v_2:=13231 \rightarrow  12321 \rightarrow  13231 \rightarrow  31231 \rightarrow  31213 \rightarrow   32123 \rightarrow  31213 \rightarrow  13213 \rightarrow  13231$.}
	\end{figure}
	
	It might seem surprising to find a counterexample in a group of such a low rank, but we need to recall that there is an infinity of paths for each graph.\\
	
	Section \ref{sect2} contains background material, notations and conventions. In Section \ref{sectFPC} we prove the FPC for all but one element in $S_4$. In Section \ref{counterexample} we present a counterexample for the FPC with a diagrammatic verification. Finally, in Section \ref{generalization} we explain (without giving a proof) 
	how to generate a family of counterexamples. We also propose a refined conjecture for the longest element of any symmetric group.\\
	
	{\it Acknowledgments: I would like to thank my advisor Nicol\'as Libedinsky for posing this problem, for many helpful discussions and valuable comments, and I would like to thank Ben Elias, for helpful discussions and feedback.}\\ 	
	
	{\it This paper was partially funded by FONDECYT project number 1200061, and by ANID scholarship number 21171339.}
	
	
	\section{Background}\label{sect2}
	
	\subsection{Rex graphs}
	
	For $n \in \mathbb{N}$, let $(W,S)$ be the Coxeter system with $W=S_n$ the symmetric group on $\{1,\ldots,n\}$, and the set of generators $S=\{s_i\ |\ i=1, 2, \ldots , n-1\}$ where each $s_i$ is the transposition $(i\ i+1)$. They are also known as {\it simple reflections}. For $s,t \in S$, $m_{st}$ is the order of $st$ (it can be $2$, or $3$ if $s\neq t$). Let $l:W\rightarrow \mathbb{Z}_{\geq0}$, be the length function and $w_{0,n}$ the longest element of $W$. If no confusion is possible, we write $w_0$ instead of $w_{0,n}$.
	
	\begin{notation}
		When no confusion is possible, we denote by $i$ the simple reflection $s_i$.
	\end{notation}
	
	\begin{example}
		The longest elements in $S_n$ for $n=2,3,4$ are respectively $1, 121, 121321$. We can obtain $w_{0,n+1}$ inductively, by joining the sequence $n(n-1)\ldots 21$ on the right of $w_{0,n}$.
	\end{example}

	\begin{definition}\label{defrexgraph}
		The {\it reduced expression graph} of an element $w \in W$, usually abreviated {\it rex graph} and denoted by $Rex(w)$, is the graph defined as follows. Its vertices are the reduced expressions of $w$, with an edge between two reduced expressions if they differ by a single braid relation. These relations are $s_is_{i+1}s_i = s_{i+1}s_is_{i+1}$ for all $i\in \{1,2\ldots, n-2\}$ and $s_is_j = s_js_i$ when $|i-j|\geq2$. We call the edges determined by the former identity {\it adjacent edges}, and the ones determined by the latter, {\it distant edges}.
	\end{definition}
	
	\begin{example}
		The reduced expression graph of $21321$.
		\begin{figure}[H]
			\centering
			\tikzset{every picture/.style={line width=0.75pt}} 
			\begin{tikzpicture}[x=0.4pt,y=0.4pt,yscale=-1,xscale=1]
				
				\draw [line width=1.5]    (160,17.5) -- (190,17.5) ;
				\draw [line width=1.5]    (360,17.5) -- (390,17.5) ;
				\draw [line width=1.5]    (60,17.5) -- (90,17.5) ;
				\draw [line width=1.5]    (260,17.5) -- (290,17.5) ;
				
				\draw (0,10) node [anchor=north west][inner sep=0.75pt]  {\scriptsize $21321$};
				\draw (95,10) node [anchor=north west][inner sep=0.75pt]   {\scriptsize $23121$};
				\draw (295,10) node [anchor=north west][inner sep=0.75pt]  {\scriptsize $32312$};
				\draw (195,10) node [anchor=north west][inner sep=0.75pt]   {\scriptsize $23212$};
				\draw (395,10) node [anchor=north west][inner sep=0.75pt]  {\scriptsize $32132$};

			\end{tikzpicture}
			\caption{}
		\end{figure}
	\end{example}
	
	\begin{definition}
		Given a rex graph of $w\in W$, we can draw the distant edges with dashed lines. With this convention, we name this colored graph the {\it expanded expressions graph of $w$}. We symbolize it by $\widetilde{\Gamma}_w$.
	\end{definition}
	
	\begin{example}\label{expandedexpressiongraph}
		The expanded expressions graph of $12321$.
		\begin{figure}[H]
			\centering
			\tikzset{every picture/.style={line width=0.75pt}} 
			\begin{tikzpicture}[x=0.4pt,y=0.4pt,yscale=-1,xscale=1]
				
				\draw [line width=2.25]  [dash pattern={on 2.53pt off 3.02pt}]  (160,40) -- (190,10) ;
				\draw [line width=1.5]    (60,50) -- (95,50) ;
				\draw [line width=2.25]  [dash pattern={on 2.53pt off 3.02pt}]  (260,10) -- (290,40) ;
				\draw [line width=1.5]    (360,50) -- (395,50) ;
				\draw [line width=2.25]  [dash pattern={on 2.53pt off 3.02pt}]  (260,90) -- (290,60) ;
				\draw [line width=2.25]  [dash pattern={on 2.53pt off 3.02pt}]  (160,60) -- (190,90) ;
				
				\draw (0,40) node [anchor=north west][inner sep=0.75pt]    {\scriptsize $12321$};
				\draw (195,0) node [anchor=north west][inner sep=0.75pt]    {\scriptsize $13213$};
				\draw (300,40) node [anchor=north west][inner sep=0.75pt]  {\scriptsize $31213$};
				\draw (100,40) node [anchor=north west][inner sep=0.75pt]  {\scriptsize $13231$};
				\draw (400,40) node [anchor=north west][inner sep=0.75pt]  {\scriptsize $32123$};
				\draw (195,85) node [anchor=north west][inner sep=0.75pt]  {\scriptsize $31231$};
			\end{tikzpicture}
			\caption{}
		\end{figure}
	\end{example}
	
	\begin{example}\label{exampleZam}
		The expanded expressions graph of $w_{0,4}$.
		\begin{figure}[H]
			\centering
			\tikzset{every picture/.style={line width=0.75pt}} 
			\begin{tikzpicture}[x=0.45pt,y=0.45pt,yscale=-1,xscale=1]
				
				\draw [line width=2.25]    (355,20) -- (375,40) ;
				\draw [line width=2.25]    (400,65) -- (413.5,78.5) -- (420,85) ;
				\draw [line width=2.25]    (105,185) -- (118.5,198.5) -- (125,205) ;
				\draw [line width=2.25]    (150,230) -- (163.5,243.5) -- (170,250) ;
				\draw [line width=2.25]    (375,230) -- (355,250) ;
				\draw [line width=2.25]    (175,20) -- (155,40) ;
				\draw [line width=2.25]    (125.33,65.33) -- (105,85) ;
				\draw [line width=2.25]  [dash pattern={on 2.53pt off 3.02pt}]  (230,10) -- (300,10) ;
				\draw [line width=2.25]  [dash pattern={on 2.53pt off 3.02pt}]  (106,105) -- (125,125) ;
				\draw [line width=2.25]  [dash pattern={on 2.53pt off 3.02pt}]  (30,125) -- (51,104) ;
				\draw [line width=2.25]  [dash pattern={on 2.53pt off 3.02pt}]  (30,145) -- (49,165) ;
				\draw [line width=2.25]  [dash pattern={on 2.53pt off 3.02pt}]  (105,165) -- (126,144) ;
				\draw [line width=2.25]  [dash pattern={on 2.53pt off 3.02pt}]  (400,125) -- (421,104) ;
				\draw [line width=2.25]  [dash pattern={on 2.53pt off 3.02pt}]  (475,165) -- (496,144) ;
				\draw [line width=2.25]  [dash pattern={on 2.53pt off 3.02pt}]  (476,105) -- (495,125) ;
				\draw [line width=2.25]  [dash pattern={on 2.53pt off 3.02pt}]  (400,145) -- (419,165) ;
				\draw [line width=2.25]    (424,189) -- (404,209) ;
				\draw [line width=2.25]  [dash pattern={on 2.53pt off 3.02pt}]  (230,265) -- (300,265) ;
				
				\draw (165,0) node [anchor=north west][inner sep=0.75pt]   [align=left] {\scriptsize 121321};
				\draw (300,0) node [anchor=north west][inner sep=0.75pt]   [align=left] {\scriptsize123121};
				\draw (96,45) node [anchor=north west][inner sep=0.75pt]   [align=left] {\scriptsize212321};
				\draw (45,85) node [anchor=north west][inner sep=0.75pt]   [align=left] {\scriptsize213231};
				\draw (45,170) node [anchor=north west][inner sep=0.75pt]   [align=left] {\scriptsize231213};
				\draw (120,127) node [anchor=north west][inner sep=0.75pt]   [align=left] {\scriptsize231231};
				\draw (-30,127) node [anchor=north west][inner sep=0.75pt]   [align=left] {\scriptsize213213};
				\draw (95,210) node [anchor=north west][inner sep=0.75pt]   [align=left] {\scriptsize232123};
				\draw (160,256) node [anchor=north west][inner sep=0.75pt]   [align=left] {\scriptsize323123};
				\draw (375,45) node [anchor=north west][inner sep=0.75pt]   [align=left] {\scriptsize123212};
				\draw (420,85) node [anchor=north west][inner sep=0.75pt]   [align=left] {\scriptsize132312};
				\draw (340,127) node [anchor=north west][inner sep=0.75pt]   [align=left] {\scriptsize132132};
				\draw (495,127) node [anchor=north west][inner sep=0.75pt]   [align=left] {\scriptsize312312};
				\draw (420,170) node [anchor=north west][inner sep=0.75pt]   [align=left] {\scriptsize312132};
				\draw (370,210) node [anchor=north west][inner sep=0.75pt]   [align=left] {\scriptsize321232};
				\draw (302,256) node [anchor=north west][inner sep=0.75pt]   [align=left] {\scriptsize321323};
				
			\end{tikzpicture}
			\caption{}\label{exp3.5}
		\end{figure}
	\end{example}
	
	There are different kinds of cycles appearing in Figure \ref{exp3.5}. For instance, a square is formed between 213231 and 231213, because there are two {\it disjoint} distant moves connecting them. In other words, these movements can be applied in either order. Any square of this kind in any graph is called a {\it disjoint square}. A disjoint square can involve distant or adjacent edges. For example, there is a disjoint square of adjacent edges from 121343 to 212434. See Example \ref{disjointsquare}.
	
	\begin{example}\label{distantOctagon}
		The expanded expressions graph of $1214$, a {\it distant Octagon}.
		\begin{figure}[H]
			\centering		
			\tikzset{every picture/.style={line width=0.75pt}} 
			
			\begin{tikzpicture}[x=0.65pt,y=0.65pt,yscale=-1,xscale=1]
				
				\draw [line width=2.25]  [dash pattern={on 2.53pt off 3.02pt}]  (50,30) -- (90,30) ;
				\draw [line width=2.5]    (25,40) -- (25,70) ;
				\draw [line width=2.5]    (325,40) -- (325,70) ;
				\draw [line width=2.25]  [dash pattern={on 2.53pt off 3.02pt}]  (250,30) -- (290,30) ;
				\draw [line width=2.25]  [dash pattern={on 2.53pt off 3.02pt}]  (150,30) -- (190,30) ;
				\draw [line width=2.25]  [dash pattern={on 2.53pt off 3.02pt}]  (52,80) -- (92,80) ;
				\draw [line width=2.25]  [dash pattern={on 2.53pt off 3.02pt}]  (252,80) -- (292,80) ;
				\draw [line width=2.25]  [dash pattern={on 2.53pt off 3.02pt}]  (152,80) -- (192,80) ;
				
				\draw (5,22.4) node [anchor=north west][inner sep=0.75pt]  [font=\normalsize]  {$1214$};
				\draw (201,22.4) node [anchor=north west][inner sep=0.75pt]  [font=\normalsize]  {$1421$};
				\draw (301,22.4) node [anchor=north west][inner sep=0.75pt]  [font=\normalsize]  {$4121$};
				\draw (101,22.4) node [anchor=north west][inner sep=0.75pt]  [font=\normalsize]  {$1241$};
				\draw (6,72.4) node [anchor=north west][inner sep=0.75pt]  [font=\normalsize]  {$2124$};
				\draw (203,72.4) node [anchor=north west][inner sep=0.75pt]  [font=\normalsize]  {$2412$};
				\draw (303,72.4) node [anchor=north west][inner sep=0.75pt]  [font=\normalsize]  {$4212$};
				\draw (103,72.4) node [anchor=north west][inner sep=0.75pt]  [font=\normalsize]  {$2142$};
				
			\end{tikzpicture}
			\caption{$\widetilde{\Gamma}_{1214}$} 
		\end{figure}
	\end{example}
	
	\subsection{Bott-Samelson bimodules}\label{BSbimodules}
	
	Let $R$ be the polynomial ring over $\mathbb{R}$ in variables $x_1,\ldots, x_n$, together with an action of $W$ where $s_i$ permutes the variables $x_i$ and $x_{i+1}$. The ring $R$ is graded with $deg(x_i) = 2$. If $M = \bigoplus M^i$ is a graded $R$-module, then the grading shift convention will be $M(i)^j = M^{i+j}$. 
	For $s$ in $S$, we denote by $R^{s}$ the subring of $R$ consisting of polynomials invariant under the action of $s$. Let $B_s$ denote the graded $R$-bimodule $B_s := R \otimes_{R^s} R(1)$. The {\it Bott-Samelson bimodule} related to an expression $\underline{w}=(s, r\ldots, t)$, and denoted by $B_{\underline{w}}$, is the graded $R$-bimodule given by the tensor product of bimodules $B_{\underline{w}} = B_{s}\otimes_RB_{r}\otimes_R \ldots \otimes_R B_{t}$. Direct sums of shifts of Bott-Samelson bimodules form the full monoidal graded subcategory of $R$-bimodules denoted by $\mathbb{BS}$Bim. We simplify the notation by writing $B_{i}$ instead of $B_{s_i}$. For tensor products, we write $B_i B_j$ instead of $B_{s_i}\otimes_RB_{s_j}$.\\ 
	
	We refer to subsets of $S$ as {\it parabolic subsets}. Given such a subset, $J \subset S$, we let $R^J$ denote the subring of $R$ of polynomials invariant under all the simple reflections in $J$. The full (additive monoidal graded) subcategory of $R$-bimodules additively generated by all the shifts of direct summands of Bott-Samelson bimodules is the category $\mathbb{S}$Bim of Soergel bimodules. Soergel proved in \cite[Section 6]{soergel2007kazhdan}, that the isomorphism classes of indecomposable Soergel bimodules (up to grading shift) are parameterized by $W$. 
	The indecomposable bimodule $B_w$ appears as a summand inside $B_{s}B_{r} \ldots B_{t}$ in any reduced expression $sr\cdots t$ of $w$, and does not appear in any Bott-Samelson bimodule associated to any other element smaller than $w$ in the Bruhat order. 
	Let $B_J$ be the $R$-bimodule $B_J := R \otimes_{R^J} R$, and let $w_J$ be the longest element of the parabolic subgroup generated by $J$. It is possible to show that $B_{w_J}\cong B_J$ (see \cite[Theorem 1.4]{williamson2011singular}). Thus $B_J$ will appear as a summand of $B_{s}B_{r} \ldots B_{t}$  whenever $sr\cdots t$  is a reduced expression for $w_J$, the longest element of $J$.
	
	\subsection{Braid morphisms $f_{sr}$}\label{fsr}
	
	Consider the bimodules $X_{sr} = B_sB_rB_s\ldots$ and $X_{rs} = B_rB_sB_r\ldots$, each product having $m_{sr}$ terms. 
	We write $1^{\otimes}$ for $1\otimes 1 \otimes \ldots \otimes 1 \in R\otimes_{R^s} R \otimes_{R^r} \ldots \otimes_{R^t}R $. The morphism $f_{sr}$ is defined as the only degree 0 morphism from $X_{sr}$ to $X_{rs}$ sending $1^{\otimes}$ to $1^{\otimes}$.  
	We write $f_{s_is_j}$ as $f_{ij}$.\\ 
	
	We describe these maps in terms of certain generators (as an $(R,R)$-bimodule) of the corresponding Bott-Samelson bimodules.
	There are three cases to consider:\\ 
	
	{\bf First case:} If $|i-j|\geq 2$. The morphism $f_{ij} : B_iB_j \rightarrow B_jB_i$ is determined by the formula $f_{ij}(1^{\otimes}) = 1^{\otimes}$, because $1^{\otimes}$ generates $B_iB_j$ as a bimodule.\\
	
	{\bf Second case:} The morphism $f_{i(i+1)} : B_iB_{i+1}B_i \rightarrow B_{i+1}B_iB_{i+1}$ is determined by the formulae $f_{i(i+1)}(1^{\otimes}) = 1^{\otimes}$
	and
	\begin{equation}\label{eq1}
		f_{i(i+1)}(1{\otimes}x_i{\otimes}1{\otimes}1) = (x_i+x_{i+1}){\otimes}1{\otimes}1{\otimes}1 - 1{\otimes}1{\otimes}1{\otimes}x_{i+2}.
	\end{equation} 
	
	{\bf Third case:} The morphism $f_{i(i-1)} : B_iB_{i-1}B_i \rightarrow B_{i-1}B_iB_{i-1}$ is determined by the formulae $f_{i(i-1)}(1^{\otimes}) = 1^{\otimes}$
	and
	\begin{equation}\label{eq2}
		f_{i(i-1)}(1{\otimes}x_{i+1}{\otimes}1{\otimes}1) = 1{\otimes}1{\otimes}1{\otimes}(x_i+x_{i+1}) - x_{i-1}{\otimes}1{\otimes}1{\otimes}1.
	\end{equation}

	\subsection{Path morphisms}
		Let $G = (V, E, \varphi)$ be a graph. Here, $V$ denotes the set of vertices, $E$ denotes the set of edges, and $\varphi$ is the incidence function $\displaystyle \varphi :E\to \{\{x,y\}\mid x,y\in V\;{\textrm {and}}\;x\neq y\}$.
	
		\begin{definition}\label{path}
			Let $G = (V, E, \varphi)$ be a graph. A ${\it path}$ $p$ is a sequence of edges $(e_1, e_2, \ldots, e_{n-1})$ for which there is a sequence of vertices $[v_1, v_2,\ldots, v_n]$ such that $\varphi(e_i) = \{v_i, v_{i + 1}\}$ for $i = 1, 2, \ldots, n-1$. The sequence $[v_1, v_2,\ldots, v_n]$ is the \textit{vertex sequence} of the path. Note that it is possible to recover the edges of a path from its vertex sequence, so we will work with vertex sequences and paths indistinctly. For any path $p$ we denote by $[p]$ the associated sequence of vertices. We say that the \textit{length} of $p$ is $n$.
		\end{definition}
	
	\begin{definition}\label{semiorientation}
		We give a semi-orientation to the rex graph. We orient adjacent edges with the lexicographic order, so these edges go from $i(i + 1)i$ to $(i + 1)i(i + 1)$. The distant edges remain unoriented. When we speak of an {\it oriented path} in a semi-oriented graph, we refer to a path which may follow unoriented edges freely, but can only follow oriented edges along the orientation. A {\it reverse-oriented path} is a path oriented backwards. When we say {\it path} with no specification, we refer to any path. The starting point (vertex) and the ending point of a path $p$ will be referred as $p_a$ and $p_z$ respectively. Here, a {\it subpath} is a path that makes up part of a larger path.
	\end{definition}
	
		For a pair of Bott-Samelson bimodules $B, B'$ whose expressions differ by a single braid relation, we have a morphism of the type $$\mathrm{Id} \otimes \ldots \otimes f_{sr}\otimes \ldots \otimes \mathrm{Id} \in \mathrm{Hom}(B,B')$$ where $s$ and $r$ depend on the aforementioned braid relation. 
	\begin{example}
		In $S_4$, the expressions $212321$ and $213231$ are reduced expressions of the same element, and they differ by the braid relation $232=323$. The aforementioned morphism from $212321$ to $213231$ has the following form.
		\begin{equation*}
			\mathrm{Id}^2\otimes f_{23}\otimes \mathrm{Id} \colon B_{2} B_{1}(B_{2} B_{3} B_{2}) B_{1} \rightarrow B_{2} B_{1} (B_{3} B_{2} B_{3}) B_{1}.
		\end{equation*}
	\end{example}
	
	\begin{definition}
		For each path $p$ in the rex graph $Rex(w)$ we call $f(p)$ the associated morphism between the Bott-Samelson bimodules $B_{p_a}$ and $B_{p_z}$. We call $f(p)$ a {\it path morphism}.
	\end{definition}
	
	Note that for expressions related by distant edges (first case), the morphisms $f_{sr}$ are isomorphisms. We will see in Section \ref{distantEdgesIdentification} that the path morphism associated to a composition of distant edges only depends on the starting point and the ending point. In this way, we can collapse the dashed lines obtaining a new graph that we now define. 
	
	\begin{definition}\label{confgraph}
		The {\it conflated expression graph}, denoted by $\Gamma_w$, is the quotient of $\widetilde{\Gamma}_w$ (or $Rex(w)$) by all its distant edges. In other words, if $p$ is a path such that all its edges are distant, then we identify $p_a$ and $p_z$. We remark that there are no possible adjacent edges between $p_a$ and $p_z$ because the sum of the indices of a reduced expression remains unchanged when applied to a distant edge and varies when applied to an adjacent edge. 
		When identifying the vertices we must choose a representative, which usually will be a specific one depending on the path morphisms we are working with. When the representative is not explicit,  by convention, we will consider that it is the lower in the lexicographical order among the identified elements. We remark that there might be multiple edges between two vertices in this graph (see Example \ref{distantOctagon}), as opposed to the expanded expressions graph. Here we choose a representative following the same criteria, avoiding multigraphs.
	\end{definition}
		
		\begin{notation}\label{notation1}
			If $e$ is an edge (resp. $v$ is a vertex) of the expanded expressions graph we call $\pi(e)$ (resp. $\pi(v)$) its image in the conflated expression graph. In particular, if $e$ is a distant edge, $\pi(e)=\emptyset.$ For a sequence of edges $p=(e_1,\ldots, e_n)$ we denote by $\pi(p)$ the sequence $(\pi(e_1),\ldots, \pi(e_n))$ omitting $\pi(e_j)$ when it is empty. 
		\end{notation}
	
	\begin{example}\label{disjointsquare}
		The following figure is the conflated expression graph of 121343. This configuration is also known as {\it disjoint square}.
		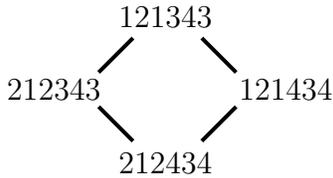
\begin{figure}[H]
			\centering
			\tikzset{every picture/.style={line width=0.75pt}} 
			\begin{tikzpicture}[x=0.65pt,y=0.65pt,yscale=-1,xscale=1]
				
				\draw [line width=1.5]    (50,60) -- (70,40) ;
				\draw [line width=1.5]    (110,100) -- (130,80) ;
				\draw [line width=1.5]    (110,40) -- (130,60) ;
				\draw [line width=1.5]    (50,80) -- (70,100) ;
				
				\draw (60,20) node [anchor=north west][inner sep=0.75pt]   [font=\normalsize]  {$121343$};
				\draw (-5,62) node [anchor=north west][inner sep=0.75pt]  [font=\normalsize] {$212343$};
				\draw (130,62) node [anchor=north west][inner sep=0.75pt]  [font=\normalsize]  {$121434$};
				\draw (60,105) node [anchor=north west][inner sep=0.75pt]  [font=\normalsize]  {$212434$};
				
			\end{tikzpicture}
			\caption{$\Gamma_{121343}$}	
		\end{figure}
	\end{example}
	
	\begin{example}
		The conflated expression graph for 12321 in $S_4$ has three vertices.
		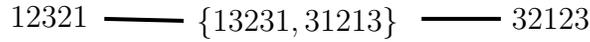
\begin{figure}[H]
			\centering
			\tikzset{every picture/.style={line width=0.75pt}}  
			\begin{tikzpicture}[x=0.75pt,y=0.75pt,yscale=-1,xscale=1]
				
				\draw [color={rgb, 255:red, 0; green, 0; blue, 0 }  ,draw opacity=1 ][line width=1.5]    (50,10) -- (90,11) ;
				\draw [color={rgb, 255:red, 0; green, 0; blue, 0 }  ,draw opacity=1 ][line width=1.5]    (210,10) -- (250,10) ;
				
				\draw (1,2.4) node [anchor=north west][inner sep=0.75pt]  [font=\normalsize]  {$12321$};
				\draw (95,2.4) node [anchor=north west][inner sep=0.75pt]  [font=\normalsize]  {$\{13231,31213\}$};
				\draw (254,3.4) node [anchor=north west][inner sep=0.75pt]  [font=\normalsize]  {$32123$};
				
			\end{tikzpicture}
			\caption{$\Gamma_{12321}$}\label{conflatedexpressiongraph}
		\end{figure}
	\end{example}
		
	\begin{example}
		The expanded expressions graph for 246 in $S_7$ and its conflated expression graph. A configuration like the first one is known as {\it distant hexagon}.
		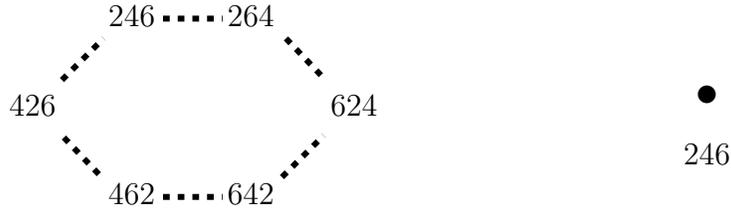
\begin{figure}[H]
			\centering
			\tikzset{every picture/.style={line width=0.75pt}} 
			\begin{tikzpicture}[x=0.75pt,y=0.75pt,yscale=-1,xscale=1]
				
				\draw [line width=2.25]  [dash pattern={on 2.53pt off 3.02pt}]  (110,20) -- (140,20) ;
				\draw [line width=2.25]  [dash pattern={on 2.53pt off 3.02pt}]  (172,30) -- (191,50) ;
				\draw [line width=2.25]  [dash pattern={on 2.53pt off 3.02pt}]  (110,110) -- (140,110) ;
				\draw [line width=2.25]  [dash pattern={on 2.53pt off 3.02pt}]  (59,51) -- (80,30) ;
				\draw [line width=2.25]  [dash pattern={on 2.53pt off 3.02pt}]  (60,80) -- (79,100) ;
				\draw [line width=2.25]  [dash pattern={on 2.53pt off 3.02pt}]  (170,100) -- (191,79) ;
				
				\draw (81,12) node [anchor=north west][inner sep=0.75pt]   [align=left] {246};
				\draw (141,12) node [anchor=north west][inner sep=0.75pt]   [align=left] {264};
				\draw (31,57) node [anchor=north west][inner sep=0.75pt]   [align=left] {426};
				\draw (81,102) node [anchor=north west][inner sep=0.75pt]   [align=left] {462};
				\draw (193,57) node [anchor=north west][inner sep=0.75pt]   [align=left] {624};
				\draw (141,102) node [anchor=north west][inner sep=0.75pt]   [align=left] {642};
				\draw (377,52.4) node [anchor=north west][inner sep=0.75pt] [font=\Large,color={rgb, 255:red, 0; green, 0; blue, 0 }  ,opacity=1 ]    {$\bullet $};
				\draw (371,82) node [anchor=north west][inner sep=0.75pt]   [align=left] {246};
				
			\end{tikzpicture}
			
			\caption{$\widetilde{\Gamma}_{246}$ and $\Gamma_{246}$}\label{DistantHexagon}
		\end{figure}
	\end{example}
	
	\begin{definition}\label{maninschechtman}
		Considering the semiorientation in Definition \ref{semiorientation}, and the quotient in Definition \ref{confgraph}, we obtain a proper orientation in $\Gamma_w$.
		This orientation is known as the {\it Manin-Schechtman orientation} \cite{manin1989arrangements}. 
	\end{definition}
	
	\begin{example}\label{exampleManinSchechtman}
		The conflated expression graph of 
		$w_{0,4}$ with the Manin-Schechtman orientation. We refer to this cycle in any of its forms (i.e. in its reduced, expanded, or conflated expression graph) as a {\it Zamolodchikov cycle}.
		
		\begin{figure}[H]
			\centering
			\tikzset{every picture/.style={line width=0.75pt}} 
			
			\begin{tikzpicture}[x=0.38pt,y=0.38pt,yscale=-1,xscale=1]
				
				\draw  [color={rgb, 255:red, 0; green, 0; blue, 0 }  ,draw opacity=1 ] (311,151) -- (272.92,242.92) -- (181,281) -- (89.08,242.92) -- (51,151) -- (89.08,59.08) -- (181,21) -- (272.92,59.08) -- cycle ;
				\draw [color={rgb, 255:red, 0; green, 0; blue, 0 }  ,draw opacity=1 ][line width=2]    (181,21) -- (89.08,59.08) ;
				\draw [shift={(125.43,44.02)}, rotate = 337.5] [fill={rgb, 255:red, 0; green, 0; blue, 0 }  ,fill opacity=1 ][line width=0.08]  [draw opacity=0] (16.97,-8.15) -- (0,0) -- (16.97,8.15) -- cycle    ;
				\draw [color={rgb, 255:red, 0; green, 0; blue, 0 }  ,draw opacity=1 ][line width=2]    (89.08,59.08) -- (51,151) ;
				\draw [shift={(66.06,114.65)}, rotate = 292.5] [fill={rgb, 255:red, 0; green, 0; blue, 0 }  ,fill opacity=1 ][line width=0.08]  [draw opacity=0] (16.97,-8.15) -- (0,0) -- (16.97,8.15) -- cycle    ;
				\draw [color={rgb, 255:red, 0; green, 0; blue, 0 }  ,draw opacity=1 ][line width=2]    (311,151) -- (272.92,242.92) ;
				\draw [shift={(287.98,206.57)}, rotate = 292.5] [fill={rgb, 255:red, 0; green, 0; blue, 0 }  ,fill opacity=1 ][line width=0.08]  [draw opacity=0] (16.97,-8.15) -- (0,0) -- (16.97,8.15) -- cycle    ;
				\draw [color={rgb, 255:red, 0; green, 0; blue, 0 }  ,draw opacity=1 ][line width=2]    (272.92,242.92) -- (181,281) ;
				\draw [shift={(217.35,265.94)}, rotate = 337.5] [fill={rgb, 255:red, 0; green, 0; blue, 0 }  ,fill opacity=1 ][line width=0.08]  [draw opacity=0] (16.97,-8.15) -- (0,0) -- (16.97,8.15) -- cycle    ;
				\draw [color={rgb, 255:red, 0; green, 0; blue, 0 }  ,draw opacity=1 ][line width=2]    (181,21) -- (272.92,59.08) ;
				\draw [shift={(236.57,44.02)}, rotate = 202.5] [fill={rgb, 255:red, 0; green, 0; blue, 0 }  ,fill opacity=1 ][line width=0.08]  [draw opacity=0] (16.97,-8.15) -- (0,0) -- (16.97,8.15) -- cycle    ;
				\draw [color={rgb, 255:red, 0; green, 0; blue, 0 }  ,draw opacity=1 ][line width=2]    (89.08,242.92) -- (181,281) ;
				\draw [shift={(144.65,265.94)}, rotate = 202.5] [fill={rgb, 255:red, 0; green, 0; blue, 0 }  ,fill opacity=1 ][line width=0.08]  [draw opacity=0] (16.97,-8.15) -- (0,0) -- (16.97,8.15) -- cycle    ;
				\draw [color={rgb, 255:red, 0; green, 0; blue, 0 }  ,draw opacity=1 ][line width=2]    (51,151) -- (89.08,242.92) ;
				\draw [shift={(74.02,206.57)}, rotate = 247.5] [fill={rgb, 255:red, 0; green, 0; blue, 0 }  ,fill opacity=1 ][line width=0.08]  [draw opacity=0] (16.97,-8.15) -- (0,0) -- (16.97,8.15) -- cycle    ;
				\draw [color={rgb, 255:red, 0; green, 0; blue, 0 }  ,draw opacity=1 ][line width=2]    (271.92,59.08) -- (310,151) ;
				\draw [shift={(294.94,114.65)}, rotate = 247.5] [fill={rgb, 255:red, 0; green, 0; blue, 0 }  ,fill opacity=1 ][line width=0.08]  [draw opacity=0] (16.97,-8.15) -- (0,0) -- (16.97,8.15) -- cycle    ;
				
				\draw (145,-8) node [anchor=north west][inner sep=0.75pt]  [font=\scriptsize]  {$121321$};
				\draw (282,40) node [anchor=north west][inner sep=0.75pt]  [font=\scriptsize]  {$123212$};
				\draw (10,40) node [anchor=north west][inner sep=0.75pt]  [font=\scriptsize]  {$212321$};
				\draw (-25,150) node [anchor=north west][inner sep=0.75pt]  [font=\scriptsize]  {$213231$};
				\draw (317,150) node [anchor=north west][inner sep=0.75pt]  [font=\scriptsize]  {$132312$};
				\draw (10,250) node [anchor=north west][inner sep=0.75pt]  [font=\scriptsize]  {$232123$};
				\draw (282,250) node [anchor=north west][inner sep=0.75pt]  [font=\scriptsize]  {$321232$};
				\draw (145,290) node [anchor=north west][inner sep=0.75pt]  [font=\scriptsize]  {$323123$};
				\draw (168,8) node [anchor=north west][inner sep=0.75pt]  [font=\Large,color={rgb, 255:red, 0; green, 0; blue, 0 }  ,opacity=1 ]  {$\bullet $};
				\draw (294,138) node [anchor=north west][inner sep=0.75pt]  [font=\Large,color={rgb, 255:red, 0; green, 0; blue, 0 }  ,opacity=1 ]  {$\bullet  $};
				\draw (76,46) node [anchor=north west][inner sep=0.75pt]  [font=\Large]  {$\bullet  $};
				\draw (40,138) node [anchor=north west][inner sep=0.75pt]  [font=\Large,color={rgb, 255:red, 0; green, 0; blue, 0 }  ,opacity=1 ]  {$\bullet $};
				\draw (258,46) node [anchor=north west][inner sep=0.75pt]  [font=\Large]  {$\bullet $};
				\draw (74,228) node [anchor=north west][inner sep=0.75pt]  [font=\Large]  {$\bullet $};
				\draw (168,266) node [anchor=north west][inner sep=0.75pt]  [font=\Large,color={rgb, 255:red, 0; green, 0; blue, 0 }  ,opacity=1 ]  {$\bullet $};
				\draw (260,228) node [anchor=north west][inner sep=0.75pt]  [font=\Large]  {$\bullet $};
			\end{tikzpicture}
			\caption{Manin-Schechtman oriented Zamolodchikov cycle.}\label{confZamCycle}
		\end{figure}
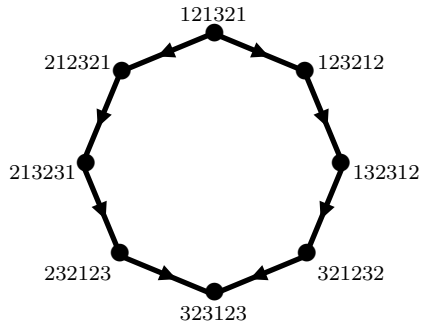
	\end{example}
	
	For any $w\in S_n$, the Manin-Schechtman orientation determines a unique source and a unique sink in $\Gamma_w$. 
	We refer to them as $\bf s$ and $\bf t$ respectively. In \cite[Section 3]{elias2016thicker}, it is proven that the Manin-Schechtman orientation satisfies the following properties.
	
	\begin{enumerate}
		\item \label{BS-Consistency} It is {\it BS-consistent} or {\it consistent with Bott-Samelson bimodules} (\cite[Prop. 3.10]{elias2016thicker}). This means that for any pair of oriented (or reverse-oriented) paths $p$ and $q$, with 
		$p_a=q_a$ and $p_z=q_z$, we have $f(p)=f(q)$.
		\item For $w_{0,n}$, the orientation is said to be {\it idempotent-magical}. This means that the morphism associated to an oriented path from $\bf s$ to $\bf t$ composed with the morphism associated to a reverse-oriented path from $\bf t$ to $\bf s$ is an idempotent.
	\end{enumerate}
	
	A {\it complete path} in a graph is a path passing through every vertex of the graph at least once. 
	Recall from the introduction the Forking Path Conjecture.
	
	\begin{conjecture*}
		Let $w \in S_n$, and let $p, q$ be two complete paths in $Rex(w)$, with $p_a=q_a$ and $p_z=q_z$. Then $f(p)=f(q)$.
	\end{conjecture*}
	
	\subsection{Conflated expression graph of the longest element}
	
	We now restrict our attention to the graph $\Gamma = \Gamma_{w_{0,n}}$ with the Manin-Schechtman orientation. 
	
	\begin{definition}\label{def2.5}
		For $\mathbf{x}, \mathbf{y} \in \Gamma$, we denote $\mathbf{x}\searrow \mathbf{y}$ (resp. $\mathbf{y} \nearrow \mathbf{x}$) for some oriented (resp. reverse-oriented) path from $\mathbf{x}$ to $\mathbf{y}$ (resp. $\mathbf{y}$ to $\mathbf{x}$), presuming that one exists. We use $f_{\mathbf{x}\searrow \mathbf{y}}$ and $f_{\mathbf{y}\nearrow \mathbf{x}}$ for the induced path morphisms, which do not depend on the choice of oriented path by the $BS$-consistency.
	\end{definition}
The following is Proposition 3.16 in \cite{elias2016thicker}.
	
	\begin{proposition}
		There is a unique source $\mathbf{s}$, and a unique sink $\mathbf{t}$ in $\Gamma$. Let $m$ be the length of the shortest (not necessarily oriented) path from $\mathbf{s}$ to $\mathbf{t}$. Then every vertex lies on some oriented path $\mathbf{s} \searrow \mathbf{t}$ of length $m$, and every oriented path $\mathbf{x}\searrow \mathbf{y}$ can be extended to a length $m$ path $\mathbf{s} \searrow \mathbf{x} \searrow \mathbf{y} \searrow \mathbf{t}$.
	\end{proposition}
	
	\begin{definition}
		For $\mathbf{x}, \mathbf{y}\in \Gamma$, let $DUD_{\mathbf{x}, \mathbf{y}} = f_{\mathbf{s}\searrow \mathbf{y}} \circ f_{\mathbf{t}\nearrow \mathbf{s}} \circ f_{\mathbf{x}\searrow \mathbf{t}}$. That is, $DUD_{\mathbf{x}, \mathbf{y}}$ corresponds to any oriented path which goes from $\mathbf{x}$ down to the sink, up to the source, and down to $\mathbf{y}$. Let $UDU_{\mathbf{x}, \mathbf{y}} = f_{\mathbf{t}\nearrow \mathbf{y}} \circ f_{\mathbf{s}\searrow \mathbf{t}} \circ f_{\mathbf{x}\nearrow \mathbf{s}}$ corresponds to any path which goes from $\mathbf{x}$ up to the source, down to the sink, and up to $\mathbf{y}$.
	\end{definition}
	
	\begin{theorem} \cite[Theorem 3.18]{elias2016thicker}\label{teo3.18}
		For all $\mathbf{x}, \mathbf{y}\in \Gamma$, we have $$DUD_{\mathbf{x}, \mathbf{y}} = UDU_{\mathbf{x}, \mathbf{y}}.$$ Its image is the indecomposable object $B_{w_0}$ corresponding to the longest element of $S_n$.
	\end{theorem} 
	
	\begin{definition}\label{Zeta}
		Let $Z= f_{\mathbf{s}\searrow \mathbf{t}}$ denote the unique oriented path morphism from source to sink. Let $\overline{Z}= f_{\mathbf{t}\nearrow \mathbf{s}}$ denote the unique reverse-oriented path morphism from sink to source.
	\end{definition}
	
	Note that $DUD_{\mathbf{t},\mathbf{s}} = \overline{Z}$, $UDU_{\mathbf{s}, \mathbf{t}}= Z$, and $DUD_{\mathbf{s},\mathbf{s}} = UDU_{\mathbf{s}, \mathbf{s}} = \overline{Z}\circ Z$. Also, note that considering $\mathbf{x} = \mathbf{s}$ and $\mathbf{y}=\mathbf{t}$, Theorem \ref{teo3.18} says that
	\begin{equation}\label{ZZZ}
		Z\circ \overline{Z}\circ Z = Z.
	\end{equation} 
	Analogously, we have
	\begin{equation}\label{ZZZ2}
		\overline{Z}\circ Z\circ \overline{Z} =  \overline{Z}.
	\end{equation} 
	
	
	\section{Forking Path Conjecture in $S_4$}\label{sectFPC}
	
	\subsection{Distant edges identification}\label{distantEdgesIdentification}

		If $\underline{w}\in \Gamma_w$ is a vertex in the conflated expression graph, the set $\pi^{-1}(\underline{w})\in \widetilde{\Gamma}_w$ is called a \textit{cloud.} If $C$ is a cloud, then by definition, every two vertices in $C$ are connected by a sequence of distant edges. If we consider the statistic $N(\underline{w})$ given by adding all the indexes of the reduced expression (for example $N(s_1s_3s_2)=1+3+2=6$) we can see that the function $N$ is constant in the vertices of a cloud.
		
		\begin{definition}\label{conflatedmorphism}
			Consider any $w\in S_n$. Let $p$ be a path in the conflated expression graph $\Gamma_w$. The path morphism $f(p)$ defined by $p$ is $f(\tilde{p})$, where $\tilde{p}=(e_1, e_2, \ldots, e_{n-1})$ is any path in the expanded expressions graph $\widetilde{\Gamma}_w$ with $\tilde{p}_a=p_a$, $\tilde{p}_z=p_z,$ and such that one obtains $p$ from $\tilde{p}$ by applying $\pi$ (see Notation \ref{notation1}) to this sequence.  
		\end{definition}
		
		\begin{proposition}
			Path morphisms in conflated expression graphs are well-defined. In other words, given two paths $\tilde{p},\ \tilde{p}'$ in $\widetilde{\Gamma}_w$ satisfying the conditions in Definition $\ref{conflatedmorphism}$, we have $f(\tilde{p}) = f(\tilde{p}')$.
		\end{proposition}
		
		\begin{proof}
			Any two paths in $\widetilde{\Gamma}_w$ defining $f(p)$ will only differ on their distant edges connecting two successive adjacent edges. So, for a fixed pair of successive adjacent edges, each sequence of distant edges will have the same starting vertex and the same ending vertex. These sequences represent oriented paths, and therefore, their induced path morphisms are the same (see \cite[Prop. 3.10]{elias2016thicker}). Repeating this argument in each sequence of distant edges, we have the result.
		\end{proof}
		
		\begin{remark}\label{remarkpi}
			Definition \ref{conflatedmorphism} does not depend on $\pi$. If we change the choices of adjacent edges $e$ such that $\pi(e)=\phi$, by \cite[Prop 3.10]{elias2016thicker} we obtain the same path morphism.
		\end{remark}
		
		The next two propositions show the equivalence between working with paths in rex graphs and working with paths in conflated expression graphs. We will therefore deduce that there is an equivalence between the Forking Path Conjecture (that we call FPC in the rest of this paper) for $\Gamma_w$ and for $\widetilde{\Gamma}_w$.
		\begin{proposition}
			For any $w\in S_n$, finding paths in its conflated expression graph giving a counterexample for the FPC gives a counterexample for the FPC in its rex graph.
		\end{proposition}
		\begin{proof}
			For any $w\in S_n$, let $p,q$ be  complete paths in $\Gamma_w$, with $p_a=q_a$ and $p_z=q_z$, such that $f(p)\neq f(q)$. By definition, $f(p)$ is equal to $f(\tilde{p})$, where $\tilde{p}$ is a path in $\widetilde{\Gamma}_w$ satisfying the requirements in Definition \ref{conflatedmorphism}. Similarly for $f(q)$ and $f(\tilde{q})$. As $p$ and $q$ are complete, $\tilde{p}$ and $\tilde{q}$ pass through every cloud in $\widetilde{\Gamma}_w$. We can modify $\tilde{p}$ and $\tilde{q}$ to pass through every vertex in $\widetilde{\Gamma}_w$ as follows. Each time $\tilde{p}$ or $\tilde{q}$ passes through a cloud, add a complete closed  path in that cloud and then continue as before (this does not alter the path morphism). Let us call $\tilde{p}_0$ and $\tilde{q}_0$ the new paths in $Rex(w)$, then they are complete paths such that $f(\tilde{p}_0)\neq f(\tilde{q}_0).$
		\end{proof}
	
		\begin{proposition}
		For any $w\in S_n$, the FPC in $\Gamma_w$ implies the FPC in $Rex(w)$.
	\end{proposition}
	\begin{proof}
		Suppose that the FPC is true in $\Gamma_w$ for some element $w\in S_n$. Let $\tilde{p}, \tilde{q}$ be complete paths in $Rex(w)$, with $\tilde{p}_a=\tilde{q}_a$ and $\tilde{p}_z=\tilde{q}_z$. 
		When applying the projection $\pi$ to these paths, it is possible to obtain $\pi(e) = \phi$ for some adjacent edges of these paths. This means that these edges vanish while doing the identifications of vertices and choices of edges in the construction of $\Gamma_w$. If this is the case, it is possible to replace each of these edges $e$ at a time. We can do that replacement with a path that goes through the corresponding cloud from the same starting vertex of $e$ to the starting vertex of the edge that does not vanish, then follows that edge, and then goes back through the corresponding cloud again to the same ending vertex of $e$, returning to the original path. These local replacements (one for each vanishing edge $e$) do not alter the resulting path morphism, because the involved subpaths are oriented paths (see \cite[Prop. 3.10]{elias2016thicker}). So, by modifying the paths as so, the resulting projections will satisfy the hypothesis of the FPC in $\Gamma_w$. Therefore $f(\tilde{p}) = f(\pi(\tilde{p})) = f(\pi(\tilde{q}))= f(\tilde{q})$ and we have the FPC for $Rex(w)$.	
	\end{proof}

		So we obtain an equivalent conjecture.
	
	\begin{conjecture*}[FPC for conflated expression graphs]
		Let $w \in S_n$. Let $p, q$ be two complete paths in $\Gamma_w$, with $p_a = q_a$ and $p_z = q_z$. Then $f(p) = f(q)$.
	\end{conjecture*}
	
	\subsection{Calculating path morphisms}
	
	\begin{definition}
		Consider $w \in S_n$ and $\Gamma_w$ with the Manin-Schechtman orientation. We say that a path is {\it straight} if it goes in an oriented or reverse-oriented fashion from one vertex $\bf x$ to another vertex $\bf y$. We denote that by $\bf x \rightarrow \bf y$. We use this notation when we do not want to specify if $x\nearrow y$ or $x\searrow y$. In particular, straight paths $\bf s \searrow \bf t$ or $\bf t \nearrow \bf s$ will be called {\it direct} paths, and we denote them by the letter $d$. We say that a pair of paths $p, q$ are {\it equivalent} and we write $p \simeq q$ if $f(p)=f(q)$, i.e., if they define the same path morphism. If a direct path (resp. straight path) is a subpath of a larger path, we call it a {\it direct (resp. straight) subpath of the larger path}.
	\end{definition} 
	
	Now we restrict our attention to $\Gamma_{w_{0,n}}$. Let $p, q$ be two complete paths with $p_a= q_a$ and $p_z = q_z$, both containing a direct subpath $d$. We will show that $p \simeq q$. The main idea is to construct equivalent paths that lead us to a reduced problem, i.e., studying equivalences between a small set of paths.  We divide $p$ into three parts: the path before $d$, from $p_a$ to $d_a$, which we call the $p^{\alpha}$ {\it subpath}, the direct subpath $d$, and the path after $d$, from $d_z$ to $p_z$, which we call the $p^{\beta}$ {\it subpath}. If there exist more than one direct subpath, it does not matter which one we choose to work with. We now focus on the $p^{\alpha}$ subpath.
	
	\begin{proposition}\label{propsimplify}
		Let $p$ be a complete path in $\Gamma_{w_{0,n}}$ with a direct subpath $d$. Then $p^{\alpha}$ is equivalent to a path $p'$ of the form $p_a \nearrow \mathbf{s} \searrow \mathbf{t} \nearrow \mathbf{s} \searrow \ldots \rightarrow d_a$, or $p_a \searrow \mathbf{t} \nearrow \mathbf{s} \searrow \mathbf{t} \nearrow \ldots \rightarrow d_a$.
	\end{proposition}
	
	\begin{proof}
		We assume without loss of generality that $d_a=\bf s$. Then, $p^{\alpha}$ ends in the vertex $\mathbf{s}$, that is, $p^{\alpha}_z=\bf s$. If $p^{\alpha}=\mathrm{id}$ (i.e. the empty sequence), we are done. If not, 
		the path $p^{\alpha}$ has a straight subpath from a vertex $x_1$ to $\bf s$ (with $x_1\neq \bf s$ ) which we take maximal in this sense, i.e. the straight subpath $x_1 \nearrow \mathbf{s}$ is not contained in any larger straight subpath.\\ 
		
		If $x_1 \nearrow \mathbf{s} = p^{\alpha}$, we have the desired path $p'$; if this is not the case, there exists a vertex $x_2$ (maximal in the same sense)
		such that $x_2 \searrow x_1$ is a subpath of $p$. We have that
		\begin{equation}{\label{ext}}
			x_2\searrow x_1\nearrow \mathbf{s} \searrow \mathbf{t}
		\end{equation}
		is a subpath of the path $p$ corresponding to the end of $p^{\alpha}$, followed by $d$. Using equation (\ref{ZZZ}), we rewrite (\ref{ext}) as $x_2\searrow x_1\nearrow \mathbf{s} \searrow \mathbf{t} \nearrow \mathbf{s} \searrow \mathbf{t}$. Now, by Theorem \ref{teo3.18}, we apply $UDU_{x_1,\mathbf{t}}=DUD_{x_1,\mathbf{t}}$,  to obtain $x_2\searrow x_1\searrow \mathbf{t} \nearrow \mathbf{s} \searrow \mathbf{t}\nearrow \mathbf{s} \searrow \mathbf{t}$. Using again (\ref{ZZZ}) to simplify, we see that (\ref{ext}) has the same path morphism as $x_2\searrow \mathbf{t}\nearrow \mathbf{s} \searrow \mathbf{t}$. We now consider the subpath $\mathbf{t}\nearrow \mathbf{s}$ as our new direct subpath $d$. Using repeatedly this process we obtain the equivalent path $p'$ of the prescribed form.
	\end{proof}
	
	The same arguments work for the $\beta$ subpath, {\it mutatis mutandis}. This way, after simplifications using the identities (\ref{ZZZ}) and (\ref{ZZZ2}) if needed, from $p$ it is possible to obtain a new path $\hat{p}$ consisting of the following: a straight path from $p_a$ to $\mathbf{s}$ or $\mathbf{t}$, followed by one or two direct paths, and then a straight path from $\mathbf{s}$ or $\mathbf{t}$ to $p_z$, satisfying $f(\hat{p})=f(p)$. If $\hat{p}_a$ or $\hat{p}_z$ are $\mathbf{s}$ or $\mathbf{t}$, then $\hat{p}$ does not have the $\alpha$ or the $\beta$ 
	subpaths (both cases may occur at the same time).
	
	\begin{definition}\label{simplified}
		The path $\hat{p}$ obtained from the application of Proposition \ref{propsimplify} to $p$ will be called a {\it simplified path}. 
	\end{definition}
	
	\begin{proposition}\label{dproof}
		In $\Gamma_{w_0}$, consider any pair of simplified complete paths, $p$ and $q$, both containing a direct path $d$. 
		Suppose that $p_a=q_a$ and $p_z=q_z$. Then $f(p)$=$f(q)$.
	\end{proposition}
	\begin{proof}	
		$\bullet$ First case: $p_a=q_a=\mathbf{s}$.  If $p_z=q_z=\mathbf{t}$, then they are necessarily equivalent to $\mathbf{s} \searrow \mathbf{t}$. If $p_z=q_z=\mathbf{s}$, they will be equivalent to $\mathbf{s} \searrow \mathbf{t} \nearrow \mathbf{s}$. If $p_z=q_z=u$, with $u$ being a vertex that is not $\mathbf{s}$ or $\mathbf{t}$, we have two possibilities: $\mathbf{s} \searrow \mathbf{t} \nearrow u$ and $\mathbf{s} \searrow \mathbf{t} \nearrow \mathbf{s} \searrow u$. By Theorem \ref{teo3.18}, $UDU_{\mathbf{s}, u}= DUD_{\mathbf{s},u}$, so they are equivalent.\\
		$\bullet$ Second case: If $p_a=q_a=\mathbf{t}$ or $p_z=q_z=\mathbf{s}$ or $p_z=q_z=\mathbf{t}$, we repeat a similar analysis as in the first case and conclude that $f(p)=f(q)$.\\  
		$\bullet$ Third case: Now we study the case $p_a=q_a=u$ and  $p_z=q_z=v$, with $u$ and $v$ being vertices that are neither $\mathbf{s}$ nor $\mathbf{t}$. There are four possible cases for the paths $p, q$.
	
		\begin{enumerate}
			\item $u\nearrow \mathbf{s}\searrow \mathbf{t} \nearrow \mathbf{s} \searrow v$	
			\item $u\searrow \mathbf{t} \nearrow \mathbf{s} \searrow v$
			\item $u\nearrow \mathbf{s} \searrow \mathbf{t} \nearrow v$
			\item $u\searrow \mathbf{t} \nearrow \mathbf{s} \searrow \mathbf{t} \nearrow v$
		\end{enumerate}
		
		The equation $UDU_{u, \mathbf{s}}=DUD_{u,\mathbf{s}}$ implies that the first is equivalent to the second, $UDU_{u, v}= DUD_{u,v}$ implies that the second is equivalent to the third, and  $UDU_{u, \mathbf{t}}=DUD_{u,\mathbf{t}}$ implies that the third is equivalent to the fourth. 
	\end{proof}
	
	The FPC would follow if we could guarantee the existence of a direct subpath in any path, but this is not the case. Despite this, we will show that in $\Gamma_{w_{0,4}}$, any complete path will always contain a subpath that is equivalent to a direct path, proving the conjecture for this element. 
	
	\subsection{Diagrammatic calculus}\label{diagrammaticcalculus}
	
	We begin by considering a complete path $p$ in $\Gamma_{w_{0,4}}$ and its path morphism $f(p)$. Since the path is complete, the vertices $\mathbf{s}$ and $\mathbf{t}$ are part of the path.
	Note that it could be possible to visit these points multiple times.
	So there is at least one subpath, that we will call {\it candidate path}, starting in $\mathbf{s}$  and ending in $\mathbf{t}$, or starting in $\mathbf{t}$ and ending in $\mathbf{s}$, minimal with this property. This means that there are no proper subpaths of the candidate path starting in $\mathbf{s}$ and ending in $\mathbf{t}$, or starting in $\mathbf{t}$ and ending in $\mathbf{s}$.
	\\ 
	Without loss of generality we will assume that the candidate path starts  in $\mathbf{s}$ and ends in $\mathbf{t}$.	
	Since the Zamolodchikov cycle has a ``ring shape'' (see Figure \ref{confZamCycle}), our conditions imply that the candidate path will be hosted either in the left or in the right half of this cycle. Let us suppose without loss of generality that the candidate path is in the left half. Let us consider the following path which represents the desired direct subpath.
	
	\begin{equation} \label{eq:6}
		\mathbf{s}\searrow A\searrow B \searrow C \searrow \mathbf{t}
	\end{equation}
	
	This is a path from $\mathbf{s}$ to $\mathbf{t}$ where $\mathbf{s}=121321$, $A=212321$, $B = 213231 = 231231 = 213213 = 231213$, $C=232123$, $\mathbf{t}=323123$.
	
	Recall from Def. \ref{path} that for any path $p$ we denote by $[p]$ the associated sequence of vertices. The beginning of our candidate path $k$ is from $\mathbf{s}$ to $A$. We cannot return to $\mathbf{s}$ by the minimality of the candidate path. So we have $[k]=[\mathbf{s}, A, B,\ldots, \mathbf{t}]$.  
	Once we are in $B$ we can go back to $A$ or go forward to $C$. If we go back to $A$, since we cannot return to $\mathbf{s}$, we have to return to $B$. 
	
	\begin{lemma}{\label{propSingle}}
		For ${\bf s}, A, B$ as in path (\ref{eq:6}), we have $[A, B, A, B] \simeq  [A, B]$ (so $[\mathbf{s}, A, B]\simeq [\mathbf{s}, A, B,A,B] \simeq [\mathbf{s}, A, B, A, B, A, B]$, and so on). Also, $[B, A, B, A] \simeq  [B, A]$.
	\end{lemma}
	
	\begin{proof}
		This is a well-known identity. Such composition of morphisms can be represented and decomposed as 
		illustrated in Figure \ref{fig3.2} below. We will use black rectangular frames to highlight spots where we use local relations. Blue, red, and green correspond to indexes 1, 2, and 3 respectively. The following is a consequence of  \cite[Eq. 2.27]{elias2016thicker}
		
		\begin{figure}[H]
			\centering
			\tikzset{every picture/.style={line width=1pt}} 
			\begin{tikzpicture}[x=0.3pt,y=0.3pt,yscale=-1,xscale=1]
				
				\draw [color={rgb, 255:red, 255; green, 0; blue, 0  }  ,draw opacity=1 ][line width=1.75]    (100,180) -- (80,200) ;
				\draw [color={rgb, 255:red, 126; green, 211; blue, 33 }  ,draw opacity=1 ][line width=1.75]    (100,220) .. controls (120,240) and (120,240) .. (120,260) ;
				\draw [color={rgb, 255:red, 126; green, 211; blue, 33 }  ,draw opacity=1 ][line width=1.75]    (60,220) .. controls (40,240) and (40,240) .. (40,260) ;
				\draw [color={rgb, 255:red, 255; green, 0; blue, 0 }  ,draw opacity=1 ][line width=1.75]    (80,200) -- (80,320) ;
				\draw [color={rgb, 255:red, 255; green, 0; blue, 0 }  ,draw opacity=1 ][line width=1.75]    (60,180) -- (80,200) ;
				\draw [color={rgb, 255:red, 255; green, 0; blue, 0 }  ,draw opacity=1 ][line width=1.75]    (80,320) -- (60,340) ;
				\draw [color={rgb, 255:red, 255; green, 0; blue, 0 }  ,draw opacity=1 ][line width=1.75]    (80,320) -- (100,340) -- (100,340) ;
				\draw [color={rgb, 255:red, 255; green, 0; blue, 0 }  ,draw opacity=1 ][line width=1.75]    (100,340) .. controls (120,360) and (120,360) .. (120,380) ;
				\draw [color={rgb, 255:red, 255; green, 0; blue, 0 }  ,draw opacity=1 ][line width=1.75]    (60,340) .. controls (40,360) and (40,360) .. (40,380) ;
				\draw [color={rgb, 255:red, 126; green, 211; blue, 33 }  ,draw opacity=1 ][line width=1.75]    (80,320) -- (100,300) ;
				\draw [color={rgb, 255:red, 126; green, 211; blue, 33 }  ,draw opacity=1 ][line width=1.75]    (100,300) .. controls (120,280) and (120,280) .. (120,260) ;
				\draw [color={rgb, 255:red, 126; green, 211; blue, 33 }  ,draw opacity=1 ][line width=1.75]    (80,320) -- (60,300) -- (60,300) ;
				\draw [color={rgb, 255:red, 126; green, 211; blue, 33 }  ,draw opacity=1 ][line width=1.75]    (60,300) .. controls (40,280) and (40,280) .. (40,260) ;
				\draw [color={rgb, 255:red, 255; green, 0; blue, 0 }  ,draw opacity=1 ][line width=1.75]    (100,180) .. controls (120,160) and (120,160) .. (120,140) ;
				\draw [color={rgb, 255:red, 126; green, 211; blue, 33 }  ,draw opacity=1 ][line width=1.75]    (80,140) -- (80,200) ;
				\draw [color={rgb, 255:red, 126; green, 211; blue, 33 }  ,draw opacity=1 ][line width=1.75]    (80,200) -- (100,220) -- (100,220) ;
				\draw [color={rgb, 255:red, 126; green, 211; blue, 33 }  ,draw opacity=1 ][line width=1.75]    (80,200) -- (60,220) ;
				\draw [color={rgb, 255:red, 126; green, 211; blue, 33 }  ,draw opacity=1 ][line width=1.75]    (80,320) -- (80,380) ;
				\draw [color={rgb, 255:red, 255; green, 0; blue, 0  }  ,draw opacity=1 ][line width=1.75]    (60,180) .. controls (40,160) and (40,160) .. (40,140) ;
				\draw [color={rgb, 255:red, 255; green, 0; blue, 0 }  ,draw opacity=1 ][line width=1.75]    (80,20) -- (80,80) ;
				\draw [color={rgb, 255:red, 255; green, 0; blue, 0 }  ,draw opacity=1 ][line width=1.75]    (100,100) -- (80,80) ;
				\draw [color={rgb, 255:red, 126; green, 211; blue, 33 }  ,draw opacity=1 ][line width=1.75]    (100,60) .. controls (120,40) and (120,40) .. (120,20) ;
				\draw [color={rgb, 255:red, 126; green, 211; blue, 33 }  ,draw opacity=1 ][line width=1.75]    (60,60) .. controls (40,40) and (40,40) .. (40,20) ;
				\draw [color={rgb, 255:red, 255; green, 0; blue, 0 }  ,draw opacity=1 ][line width=1.75]    (60,100) -- (80,80) ;
				\draw [color={rgb, 255:red, 255; green, 0; blue, 0 }  ,draw opacity=1 ][line width=1.75]    (100,100) .. controls (120,120) and (120,120) .. (120,140) ;
				\draw [color={rgb, 255:red, 126; green, 211; blue, 33 }  ,draw opacity=1 ][line width=1.75]    (80,140) -- (80,80) ;
				\draw [color={rgb, 255:red, 126; green, 211; blue, 33 }  ,draw opacity=1 ][line width=1.75]    (80,80) -- (100,60) -- (100,60) ;
				\draw [color={rgb, 255:red, 126; green, 211; blue, 33 }  ,draw opacity=1 ][line width=1.75]    (80,80) -- (60,60) ;
				\draw [color={rgb, 255:red, 255; green, 0; blue, 0 }  ,draw opacity=1 ][line width=1.75]    (60,100) .. controls (40,120) and (40,120) .. (40,140) ;
				\draw   (20,140) -- (140,140) -- (140,380) -- (20,380) -- cycle ;
				\draw [color={rgb, 255:red, 255; green, 0; blue, 0 }  ,draw opacity=1 ][line width=1.75]    (220,20) -- (220,80) ;
				\draw [color={rgb, 255:red, 255; green, 0; blue, 0 }  ,draw opacity=1 ][line width=1.75]    (240,100) -- (220,80) ;
				\draw [color={rgb, 255:red, 126; green, 211; blue, 33 }  ,draw opacity=1 ][line width=1.75]    (240,60) .. controls (260,40) and (260,40) .. (260,20) ;
				\draw [color={rgb, 255:red, 126; green, 211; blue, 33 }  ,draw opacity=1 ][line width=1.75]    (200,60) .. controls (180,40) and (180,40) .. (180,20) ;
				\draw [color={rgb, 255:red, 255; green, 0; blue, 0 }  ,draw opacity=1 ][line width=1.75]    (200,100) -- (220,80) ;
				\draw [color={rgb, 255:red, 255; green, 0; blue, 0 }  ,draw opacity=1 ][line width=1.75]    (240,100) .. controls (260,120) and (260,120) .. (260,140) ;
				\draw [color={rgb, 255:red, 126; green, 211; blue, 33 }  ,draw opacity=1 ][line width=1.75]    (220,140) -- (220,80) ;
				\draw [color={rgb, 255:red, 126; green, 211; blue, 33 }  ,draw opacity=1 ][line width=1.75]    (220,80) -- (240,60) ;
				\draw [color={rgb, 255:red, 126; green, 211; blue, 33 }  ,draw opacity=1 ][line width=1.75]    (220,80) -- (200,60) ;
				\draw [color={rgb, 255:red, 255; green, 0; blue, 0 }  ,draw opacity=1 ][line width=1.75]    (200,100) .. controls (180,120) and (180,120) .. (180,140) ;
				\draw [color={rgb, 255:red, 255; green, 0; blue, 0 }  ,draw opacity=1 ][line width=1.75]    (180,140) -- (180,380) ;
				\draw [color={rgb, 255:red, 126; green, 211; blue, 33}  ,draw opacity=1 ][line width=1.75]    (220,380) -- (220,140) ;
				\draw [color={rgb, 255:red, 255; green, 0; blue, 0 }  ,draw opacity=1 ][line width=1.75]    (260,140) -- (260,380) ;
				\draw [color={rgb, 255:red, 255; green, 0; blue, 0 }  ,draw opacity=1 ][line width=1.75]    (400,180) -- (380,200) ;
				\draw [color={rgb, 255:red, 255; green, 0; blue, 0 }  ,draw opacity=1 ][line width=1.75]    (380,200) -- (380,320) ;
				\draw [color={rgb, 255:red, 255; green, 0; blue, 0 }  ,draw opacity=1 ][line width=1.75]    (360,180) -- (380,200) ;
				\draw [color={rgb, 255:red, 255; green, 0; blue, 0 }  ,draw opacity=1 ][line width=1.75]    (380,320) -- (360,340) ;
				\draw [color={rgb, 255:red, 255; green, 0; blue, 0 }  ,draw opacity=1 ][line width=1.75]    (380,320) -- (400,340) -- (400,340) ;
				\draw [color={rgb, 255:red, 255; green, 0; blue, 0 }  ,draw opacity=1 ][line width=1.75]    (400,340) .. controls (420,360) and (420,360) .. (420,380) ;
				\draw [color={rgb, 255:red, 255; green, 0; blue, 0 }  ,draw opacity=1 ][line width=1.75]    (360,340) .. controls (340,360) and (340,360) .. (340,380) ;
				\draw [color={rgb, 255:red, 255; green, 0; blue, 0 }  ,draw opacity=1 ][line width=1.75]    (400,180) .. controls (420,160) and (420,160) .. (420,140) ;
				\draw [color={rgb, 255:red, 126; green, 211; blue, 33 }  ,draw opacity=1 ][line width=1.75]    (380,360) -- (380,380) ;
				\draw [shift={(380,360)}, rotate = 90] [color={rgb, 255:red, 126; green, 211; blue, 33 }  ,draw opacity=1 ][fill={rgb, 255:red, 126; green, 211; blue, 33 }  ,fill opacity=1 ][line width=1.75]      (0, 0) circle [x radius= 4.62, y radius= 4.62]   ;
				\draw [color={rgb, 255:red, 255; green, 0; blue, 0  }  ,draw opacity=1 ][line width=1.75]    (360,180) .. controls (340,160) and (340,160) .. (340,140) ;
				\draw [color={rgb, 255:red, 255; green, 0; blue, 0  }  ,draw opacity=1 ][line width=1.75]    (380,20) -- (380,80) ;
				\draw [color={rgb, 255:red, 255; green, 0; blue, 0 }  ,draw opacity=1 ][line width=1.75]    (400,100) -- (380,80) ;
				\draw [color={rgb, 255:red, 126; green, 211; blue, 33 }  ,draw opacity=1 ][line width=1.75]    (400,60) .. controls (420,40) and (420,40) .. (420,20) ;
				\draw [color={rgb, 255:red, 126; green, 211; blue, 33 }  ,draw opacity=1 ][line width=1.75]    (360,60) .. controls (340,40) and (340,40) .. (340,20) ;
				\draw [color={rgb, 255:red, 255; green, 0; blue, 0 }  ,draw opacity=1 ][line width=1.75]    (360,100) -- (380,80) ;
				\draw [color={rgb, 255:red, 255; green, 0; blue, 0 }  ,draw opacity=1 ][line width=1.75]    (400,100) .. controls (420,120) and (420,120) .. (420,140) ;
				\draw [color={rgb, 255:red, 126; green, 211; blue, 33 }  ,draw opacity=1 ][line width=1.75]    (380,160) -- (380,80) ;
				\draw [shift={(380,160)}, rotate = 270] [color={rgb, 255:red, 126; green, 211; blue, 33 }  ,draw opacity=1 ][fill={rgb, 255:red, 126; green, 211; blue, 33 }  ,fill opacity=1 ][line width=1.75]      (0, 0) circle [x radius= 4.62, y radius= 4.62]   ;
				\draw [color={rgb, 255:red, 126; green, 211; blue, 33 }  ,draw opacity=1 ][line width=1.75]    (380,80) -- (400,60) ;
				\draw [color={rgb, 255:red, 126; green, 211; blue, 33 }  ,draw opacity=1 ][line width=1.75]    (380,80) -- (360,60) ;
				\draw [color={rgb, 255:red, 255; green, 0; blue, 0 }  ,draw opacity=1 ][line width=1.75]    (360,100) .. controls (340,120) and (340,120) .. (340,140) ;
				\draw   (320,20) -- (440,20) -- (440,280) -- (320,280) -- cycle ;
				
				\draw (140,200) node [anchor=north west][inner sep=0.75pt]    {$=$};
				\draw (275,195) node [anchor=north west][inner sep=0.75pt]    {$+$};
				
			\end{tikzpicture}
			\caption{}\label{fig3.2}
		\end{figure}
		
		The part inside the rectangle in the second summand in Figure \ref{fig3.2} is decomposed as in Figure \ref{fig3.3}, by  \cite[Eq. 2.26]{elias2016thicker}.
		
		\begin{figure}[H]
			\centering
			\tikzset{every picture/.style={line width=1pt}} 
			\begin{tikzpicture}[x=0.3pt,y=0.3pt,yscale=-1,xscale=1]
				
				\draw [color={rgb, 255:red, 255; green, 0; blue, 0 }  ,draw opacity=1 ][line width=1.75]    (80,180) .. controls (100,160) and (100,160) .. (100,140) ;
				\draw [color={rgb, 255:red, 255; green, 0; blue, 0 }  ,draw opacity=1 ][line width=1.75]    (40,180) .. controls (20,160) and (20,160) .. (20,140) ;
				\draw [color={rgb, 255:red, 255; green, 0; blue, 0 }  ,draw opacity=1 ][line width=1.75]    (60,20) -- (60,80) ;
				\draw [color={rgb, 255:red, 255; green, 0; blue, 0 }  ,draw opacity=1 ][line width=1.75]    (80,100) -- (60,80) ;
				\draw [color={rgb, 255:red, 126; green, 211; blue, 33 }  ,draw opacity=1 ][line width=1.75]    (80,60) .. controls (100,40) and (100,40) .. (100,20) ;
				\draw [color={rgb, 255:red, 126; green, 211; blue, 33 }  ,draw opacity=1 ][line width=1.75]    (40,60) .. controls (20,40) and (20,40) .. (20,20) ;
				\draw [color={rgb, 255:red, 255; green, 0; blue, 0 }  ,draw opacity=1 ][line width=1.75]    (40,100) -- (60,80) ;
				\draw [color={rgb, 255:red, 255; green, 0; blue, 0 }  ,draw opacity=1 ][line width=1.75]    (80,100) .. controls (100,120) and (100,120) .. (100,140) ;
				\draw [color={rgb, 255:red, 126; green, 211; blue, 33 }  ,draw opacity=1 ][line width=1.75]    (60,120) -- (60,80) ;
				\draw [shift={(60,120)}, rotate = 270.64] [color={rgb, 255:red, 126; green, 211; blue, 33 }  ,draw opacity=1 ][fill={rgb, 255:red, 126; green, 211; blue, 33 }  ,fill opacity=1 ][line width=1.75]      (0, 0) circle [x radius= 4.62, y radius= 4.62]   ;
				\draw [color={rgb, 255:red, 126; green, 211; blue, 33 }  ,draw opacity=1 ][line width=1.75]   (60,80) -- (80,60) ;
				\draw [color={rgb, 255:red, 126; green, 211; blue, 33 }  ,draw opacity=1 ][line width=1.75]    (60,80)-- (40,60) ;
				\draw [color={rgb, 255:red, 255; green, 0; blue, 0 }  ,draw opacity=1 ][line width=1.75]    (40,100) .. controls (20,120) and (20,120) .. (20,140) ;
				\draw [color={rgb, 255:red, 255; green, 0; blue, 0 }  ,draw opacity=1 ][line width=1.75]    (80,180) -- (60,200) ;
				\draw [color={rgb, 255:red, 255; green, 0; blue, 0 }  ,draw opacity=1 ][line width=1.75]    (60,200) -- (60,280) ;
				\draw [color={rgb, 255:red, 255; green, 0; blue, 0 }  ,draw opacity=1 ][line width=1.75]    (40,180) -- (60,200) ;
				\draw [color={rgb, 255:red, 255; green, 0; blue, 0  }  ,draw opacity=1 ][line width=1.75]    (320,100) .. controls (300,120) and (300,120) .. (300,140) ;
				\draw [color={rgb, 255:red, 255; green, 0; blue, 0 }  ,draw opacity=1 ][line width=1.75]    (360,100) .. controls (380,120) and (380,120) .. (380,140) ;
				\draw [color={rgb, 255:red, 255; green, 0; blue, 0  }  ,draw opacity=1 ][line width=1.75]    (220,180) .. controls (240,160) and (240,160) .. (240,140) ;
				\draw [color={rgb, 255:red, 255; green, 0; blue, 0 }  ,draw opacity=1 ][line width=1.75]    (180,180) .. controls (160,160) and (160,160) .. (160,140) ;
				\draw [color={rgb, 255:red, 126; green, 211; blue, 33 }  ,draw opacity=1 ][line width=1.75]    (320,60) .. controls (300,40) and (300,40) .. (300,20) ;
				\draw [color={rgb, 255:red, 126; green, 211; blue, 33 }  ,draw opacity=1 ][line width=1.75]    (360,60) .. controls (380,40) and (380,40) .. (380,20) ;
				\draw [color={rgb, 255:red, 255; green, 0; blue, 0 }  ,draw opacity=1 ][line width=1.75]    (220,180) -- (200,200) ;
				\draw [color={rgb, 255:red, 255; green, 0; blue, 0 }  ,draw opacity=1 ][line width=1.75]    (200,200) -- (200,280) ;
				\draw [color={rgb, 255:red, 255; green, 0; blue, 0 }  ,draw opacity=1 ][line width=1.75]    (180,180) -- (200,200) ;
				\draw [color={rgb, 255:red, 126; green, 211; blue, 33 }  ,draw opacity=1 ][line width=1.75]    (160,60) -- (160,20) ;
				\draw [shift={(160,60.17)}, rotate = 270.64] [color={rgb, 255:red, 126; green, 211; blue, 33 }  ,draw opacity=1 ][fill={rgb, 255:red, 126; green, 211; blue, 33 }  ,fill opacity=1 ][line width=1.75]      (0, 0) circle [x radius= 4.62, y radius= 4.62]   ;
				\draw [color={rgb, 255:red, 126; green, 211; blue, 33 }  ,draw opacity=1 ][line width=1.75]    (240,60) -- (240,20) ;
				\draw [shift={(240,60.17)}, rotate = 270.64] [color={rgb, 255:red, 126; green, 211; blue, 33 }  ,draw opacity=1 ][fill={rgb, 255:red, 126; green, 211; blue, 33 }  ,fill opacity=1 ][line width=1.75]      (0, 0) circle [x radius= 4.62, y radius= 4.62]   ;
				\draw [color={rgb, 255:red, 255; green, 0; blue, 0 }  ,draw opacity=1 ][line width=1.75]    (200,20) -- (200,80) ;
				\draw [color={rgb, 255:red, 255; green, 0; blue, 0 }  ,draw opacity=1 ][line width=1.75]    (220,100) -- (200,80) ;
				\draw [color={rgb, 255:red, 255; green, 0; blue, 0 }  ,draw opacity=1 ][line width=1.75]    (180,100) -- (200,80) ;
				\draw [color={rgb, 255:red, 255; green, 0; blue, 0  }  ,draw opacity=1 ][line width=1.75]    (220,100) .. controls (240,120) and (240,120) .. (240,140) ;
				\draw [color={rgb, 255:red, 255; green, 0; blue, 0  }  ,draw opacity=1 ][line width=1.75]    (180,100) .. controls (160,120) and (160,120) .. (160,140) ;
				\draw [color={rgb, 255:red, 255; green, 0; blue, 0 }  ,draw opacity=1 ][line width=1.75]    (340,60) -- (340,20) ;
				\draw [shift={(340,55)}, rotate = 270.64] [color={rgb, 255:red, 255; green, 0; blue, 0   }  ,draw opacity=1 ][fill={rgb, 255:red, 255; green, 0; blue, 0  }  ,fill opacity=1 ][line width=1.75]      (0, 0) circle [x radius= 4.62, y radius= 4.62]   ;
				\draw [color={rgb, 255:red, 126; green, 211; blue, 33 }  ,draw opacity=1 ][line width=1.75]    (320,60) .. controls (340,80) and (340,80) .. (360,60) ;
				\draw [color={rgb, 255:red, 255; green, 0; blue, 0  }  ,draw opacity=1 ][line width=1.75]    (360,100) .. controls (340,80) and (340,80) .. (320,100) ;
				\draw [color={rgb, 255:red, 255; green, 0; blue, 0  }  ,draw opacity=1 ][line width=1.75]    (360,180) .. controls (380,160) and (380,160) .. (380,140) ;
				\draw [color={rgb, 255:red, 255; green, 0; blue, 0 }  ,draw opacity=1 ][line width=1.75]    (320,180) .. controls (300,160) and (300,160) .. (300,140) ;
				\draw [color={rgb, 255:red, 255; green, 0; blue, 0 }  ,draw opacity=1 ][line width=1.75]    (360,180) -- (340,200) ;
				\draw [color={rgb, 255:red, 255; green, 0; blue, 0  }  ,draw opacity=1 ][line width=1.75]    (340,200) -- (340,280) ;
				\draw [color={rgb, 255:red, 255; green, 0; blue, 0  }  ,draw opacity=1 ][line width=1.75]    (320,180) -- (340,200) ;
				\draw  [line width=0.75]  (35,30) -- (85,30) -- (85,150) -- (35,150) -- cycle ;
				\draw (110,80) node [anchor=north west][inner sep=0.75pt]    {$=$};
				\draw (250,80) node [anchor=north west][inner sep=0.75pt]    {$+$};
				
			\end{tikzpicture}
			\caption{}
			\label{fig3.3}
		\end{figure}
		Each summand in the right-hand side is zero by \cite[Eq. 3.18 and 3.20]{elias2010diagrammatics}.\\
		Reading the diagrams upside down we conclude that $[B, A, B, A] \simeq  [B, A]$.
	\end{proof}
	
	Note that, using the same diagrams (but with different colors), we can also find the equivalence
	\begin{equation*}
		[B, C] \simeq [B, C, B, C].
	\end{equation*}
	
	Without loss of generality we will assume that our candidate path  has minimal length when compared to all its equivalent paths. Because of this, the candidate path 	
	$[\mathbf{s}, A, B,\ldots, \mathbf{t}]$ has no subsequences of the forms $[A,B,A,B]$ and $[B,C,B,C]$. So we can assume that our candidate path starts with $[\mathbf{s}, A, B, C]$. Being at $C$, if we go to $\mathbf{t}$ we are done. We will find a contradiction if this is not the case. 
	
	If we don't go to $\mathbf{t}$, the path  returns to $B$. From $B$ we can not return to $C$, because we would have $[B, C, B, C]$ as a subpath. Thus, from $B$ we go to $A$. Since we cannot return to $\mathbf{s}$ we have to go to $B$. So our path starts as follows $[\mathbf{s}, A, B, C, B, A, B]$.  Again, by minimality of the length of the candidate path, the next vertex has to be $C$. The following proposition proves the contradiction. 
	
	\begin{proposition}\label{propDoble}
		For $ A, B, C$ as in path \ref{eq:6}, we have $$[A, B, C, B, A, B, C] \simeq [A, B, C]$$ 
	\end{proposition}
	\begin{proof}
		The equivalence is proved diagrammatically in Figure \ref{fig8}.
		\begin{figure}[H]
			\centering
			\tikzset{every picture/.style={line width=0.75pt}} 

			
			\caption{}\label{fig8}
		\end{figure}
		
		The local relations that we use are all from \cite{elias2016thicker}. In particular, from $1)$ to $2)$ we apply Eq. 2.26. From $2)$ to $4)$ we apply Eq. 2.20 twice. From $3)$ we obtain $5)$ and $6)$ by means of Eq. 2.26. Applying Lemma \ref{propSingle} to the term $4)$ we obtain Figure \ref{fig3.7}. 
		Using the concluding observation of the proof of Lemma \ref{propSingle}, we find that each of the summands $5)$ and $6)$ can be rewritten as a sum of two morphisms which both are zero. Hence, $5)$ and $6)$ are both zero, as shown in Figure \ref{fig3.8}.
		\begin{figure}[H]
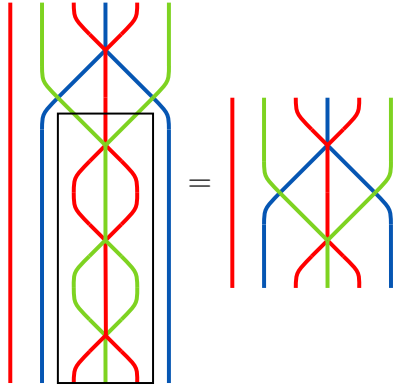

			\centering
			\tikzset{every picture/.style={line width=0.75pt}} 

			\caption{Application of Lemma \ref{propSingle}.}\label{fig3.7}
		\end{figure}
		
		\begin{figure}[H]
			\centering
			\tikzset{every picture/.style={line width=0.75pt}} 
			%
			
			\caption{}\label{fig3.8}
		\end{figure}
		From $5)$ and $6)$, to $7)$ and $8)$ we apply Eq. 2.15. From $7)$ and $8)$ to $9)$ and $10)$, we repeatedly apply Eq. 2.20 and Eq. 2.22. As in Figure \ref{fig3.3}, 
		we recognize in $9)$ and $10)$ compositions equivalent to the zero morphism.
	\end{proof}
	
	\begin{proposition}
		The Forking Path Conjecture is true for $w_{0,4}$.
	\end{proposition}
	
	\begin{proof}		
		By \ref{propSingle} and \ref{propDoble}, we have that any path starting at $\mathbf{s}$ and ending in  $\mathbf{t}$, that does not visit $\mathbf{s}$ or $\mathbf{t}$ in the rest of the path, and that is located on the left half of the Zamolodchikov cycle (See Figure \ref{confZamCycle}) is equivalent to a direct path. The same result is true if one considers paths from $\mathbf{s}$ to $\mathbf{t}$ or from $\mathbf{t}$ to $\mathbf{s}$, located on the left  or on the right half of the cycle. 
		
		The reason for this is that the proofs for the four cases will be the same as before, but turning the diagrams upside down for the case $\mathbf{t}$ to $\mathbf{s}$ in the left half, applying a vertical axial symmetry to the diagrams  for the right half for the case $\mathbf{s}$ to $\mathbf{t}$ and turning the diagrams upside down and applying a vertical axial symmetry to the diagrams for the case $\mathbf{t}$ to $\mathbf{s}$ on the right\footnote{ A deeper reason for this  is that we are implicitly applying some equivalences of categories. There  is a contravariant equivalence of monoidal categories $\mathbb{S}Bim\rightarrow \mathbb{S}Bim$, given by the flip (that sends a diagram to its horizontal flip) and  also an auto-equivalence of $\mathbb{S}Bim$  associated to the only non-trivial automorphism of the Dynkin diagram of type $A_n$, sending $s_i\mapsto s_{n-i+1}.$ }. 
		
		So we conclude that any complete path in $\Gamma_{w_{0,4}}$ has  a subpath that is equivalent to a direct path. 
		Therefore, by Proposition \ref{dproof} the proof of the FPC for $w_0$ in $S_4$ is complete.
	\end{proof}
	
	Now we verify the conjecture for the remaining elements in $S_4$ different from $12321$. The elements $w$ and their $\Gamma_w$ graphs oriented according to Manin-Schechtman  are given in the following table.
	
	\begin{figure}[H]
		\begin{center}
			\begin{tabular}{|c|c|c|c|c|c|}
				\hline
				e & 1 & 2 & 21 & 12 & 121\\
				$\bullet$ & $\bullet$ & $\bullet$ & $\bullet$ & $\bullet$ & $\bullet \rightarrow \bullet$ \\
				\hline
				\hline
				3&31&32&321&312&3121\\
				$\bullet$ & $\bullet$ & $\bullet$ & $\bullet$ & $\bullet$ & $\bullet \rightarrow \bullet$ \\
				\hline
				\hline
				23&231&232&2321&2312&23121\\
				$\bullet$ & $\bullet$ &$\bullet \rightarrow \bullet$ & $\bullet \rightarrow \bullet$  & $\bullet$ & $\bullet \rightarrow \bullet \rightarrow \bullet$ \\
				\hline
				\hline
				123&1231&1232&{\color{purple}12321}&12312&123121\\
				$\bullet$ & $\bullet \rightarrow \bullet$  & $\bullet \rightarrow \bullet$  & {\color{purple}$\bullet \rightarrow \bullet \rightarrow \bullet$ } & $\bullet \rightarrow \bullet \rightarrow \bullet$  & $\mathrm{Zam}$ \\
				\hline
			\end{tabular}
			\caption{Elements in $S_4$ and their conflated expression graphs.}\label{S4}
		\end{center}
	\end{figure}
	
	The last entry $\mathrm{Zam}$ is the Zamolodchikov cycle, as introduced in Example \ref{exampleManinSchechtman}. There is no need to verify the trivial graphs $\Gamma_w$ ($\bullet$), since the only possible morphism is the $\mathrm{id}$. For any $\bullet \rightarrow \bullet$ case, the proof of the Forking Path Conjecture follows easily from Lemma \ref{propSingle}. It remains to check the $\bullet \rightarrow \bullet \rightarrow \bullet$ cases. We will concentrate in the cases  23121 and 12312 because the other case is the one giving the counterexample to the Forking Path Conjecture. We now study the element 23121.
	
	\begin{proposition}\label{prop3.12}
		The Forking path conjecture is true for the element 23121.
	\end{proposition}
	\begin{proof}
		In this case we can speak of a \textit{simplified path} $p$ similar to that of Definition \ref{simplified}. Consider $\{x,y\}=\{\mathbf{s}, \mathbf{t}\}$. These paths will be of the form 
		$p_a\rightarrow x\rightarrow y\rightarrow p_z$, where $p_a$ could be equal to $x$, and $p_z$ could be equal to $y$, or alternatively, of the form 
		$p_a\rightarrow x\rightarrow y\rightarrow x\rightarrow p_z$, where  $p_a$ and $p_z$ could be equal to $x$. The path $p_a \rightarrow x$ (resp. $y\rightarrow p_z$ in the first case, and $x\rightarrow p_z$ in the second) will have length one only when $p_a=c$ (resp. $p_z=c$), where $c$ is the only vertex different to $\mathbf{s}$ and  $\mathbf{t}$. We consider simplified paths $p$ and $q$. 
		
		$\bullet$ We first study the case $p_a, p_z\in \{\mathbf{s}, \mathbf{t}\}$. It is immediate that $p\simeq q$, since there is only one possible path, i.e., $p=q$.
		
		$\bullet$ Consider the case $p_a=\mathbf{s}$ and $p_z=c$. There are two possible simplified paths, $P_1:=[\mathbf{s}, c, \mathbf{t} , c]$ and $P_2:=[\mathbf{s}, c, \mathbf{t} , c, \mathbf{s} , c].$ 
		We will prove that $P_1$ is equivalent to $P_2$. In the following figure,  part 1) represents the path morphism of $P_2$.
		
		\begin{figure}[H]
			\centering
			\tikzset{every picture/.style={line width=0.75pt}} 

			\caption{}\label{23121}
		\end{figure}
		We begin in $1)$ applying the relation seen in Figure \ref{fig3.2}. Then, we apply the same relation in $3)$. Diagram $5)$ is zero by the same reason as diagram $9)$ in Figure \ref{fig3.8}. In $6)$ we focus on the red and blue strands because we can retract the green dots. We have the following diagrams.
		
		\begin{figure}[H]
			\centering
			\tikzset{every picture/.style={line width=0.75pt}} 
			%
		\caption{}
		\end{figure} 
		 
		So we obtain that diagram 6) is zero. Thus, diagram 1) is equal to diagram 4), or in other words, the path morphism of $P_2$ is equal to the path morphism of $P_1.$
		
		$\bullet$ We have proved the proposition when $p_a=\mathbf{s}$ and $p_z$ is any vertex. One can prove similarly the proposition for $p_a=\mathbf{t}$ and $p_z$  any vertex, by symmetry. By flipping the diagrams we can prove the proposition for any $p_z\in \{\mathbf{s},\mathbf{t}\}$. 
		
		$\bullet$  The only case that remains to show is the equivalence for $p$ and $q$ such that $p_a=q_a=p_z=q_z=c$. There are four possible simplified paths 
		$$Q_1:=[c,\mathbf{s},c,\mathbf{t},c],  Q_2:=[c,\mathbf{t},c,\mathbf{s},c], Q_3:=[c,\mathbf{s},c,\mathbf{t},c,\mathbf{s},c], Q_4:=[c,\mathbf{t},c,\mathbf{s},c,\mathbf{t},c].$$
		In the left hand-side of the following figure, we draw the path morphism corresponding to $Q_1$, which can be seen to be equal to the path morphism corresponding to $Q_2$ after an application of the local relation (2.15) in \cite{elias2016thicker}.
		
		\begin{figure}[H]
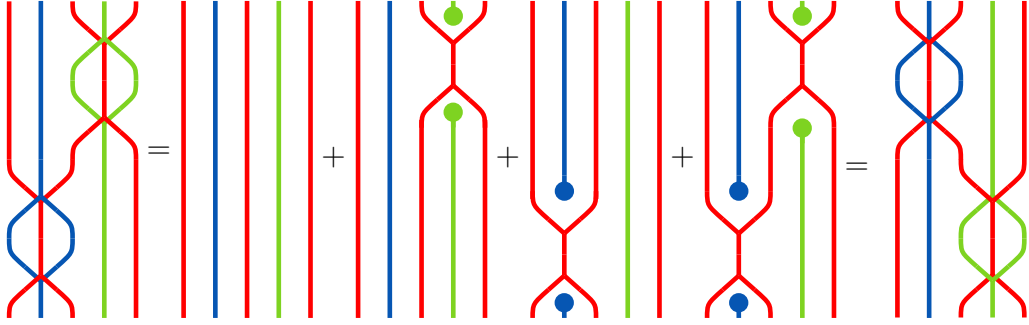

			\centering
			\tikzset{every picture/.style={line width=0.75pt}} 
			

			\caption{Proof that $f(Q_1)=f(Q_2).$}\label{fig3.12}
		\end{figure}
		By reading Figure \ref{23121} upside down, we have that $[\mathbf{s},c,\mathbf{t},c,\mathbf{s},c] \simeq [\mathbf{s},c,\mathbf{t},c]$. This way, we deduce that $Q_4 \simeq [c,\mathbf{t},c,\mathbf{s},c,\mathbf{t},c] 
		\simeq [c,\mathbf{t},c,\mathbf{s},c,\mathbf{t},c,\mathbf{s},c] \simeq [c,\mathbf{t},c,\mathbf{s},c] \simeq Q_2$. Analogously, $[c,\mathbf{s},c,\mathbf{t},c,\mathbf{s},c] \simeq [c,\mathbf{s},c,\mathbf{t},c]$, so $Q_3\simeq Q_1$.		
	\end{proof}
	
	The proof of the Forking Path Conjecture for the element $12132=12312$ is essentially the same as the proof given in Proposition \ref{prop3.12}, after applying the auto-equivalence of $\mathbb{S}Bim$ given by the unique non-trivial automorphism of the Dynkin diagram, i.e., applying a vertical symmetry to all diagrams.
	
	
	\section{The counterexample}\label{counterexample}
	
	Let us consider the element $\sigma=12321 \in S_4$. The rex graph $Rex(\sigma)$ corresponds to the following figure:
	
	\begin{figure}[H]
		\centering
		\tikzset{every picture/.style={line width=0.75pt}} 
		\begin{tikzpicture}[x=0.4pt,y=0.4pt,yscale=-1,xscale=1]
			
			\draw [line width=1.5]  
			(160,40) -- (190,10) ;
			\draw [line width=1.5]    (60,50) -- (95,50) ;
			\draw [line width=1.5]  
			(260,10) -- (290,40) ;
			\draw [line width=1.5]    (360,50) -- (395,50) ;
			\draw [line width=1.5]  
			(260,90) -- (290,60) ;
			\draw [line width=1.5]  
			(160,60) -- (190,90) ;
			
			\draw (0,40) node [anchor=north west][inner sep=0.75pt]  [font=\scriptsize]  {$12321$};
			\draw (195,0) node [anchor=north west][inner sep=0.75pt]  [font=\scriptsize]  {$13213$};
			\draw (295,40) node [anchor=north west][inner sep=0.75pt]  [font=\scriptsize]  {$31213$};
			\draw (100,40) node [anchor=north west][inner sep=0.75pt]  [font=\scriptsize]  {$13231$};
			\draw (400,40) node [anchor=north west][inner sep=0.75pt]  [font=\scriptsize]  {$32123$};
			\draw (195,85) node [anchor=north west][inner sep=0.75pt]  [font=\scriptsize]  {$31231$};
			
		\end{tikzpicture}
		\caption{Reduced expression graph of $\sigma$.}
	\end{figure}
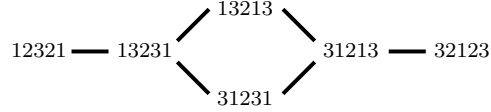
	
	Note that the four vertices in the middle are the same vertex in the conflated expression graph. 	
	Let us consider the element $$x:=1 \otimes_{s_1} 1 \otimes_{s_3} 1 \otimes_{s_2}  x_3 \otimes_{s_3} 1 \otimes_{s_1} 1$$ in the Bott-Samelson bimodule $B_{1}B_{3}B_{2}B_{3}B_{1}$. Let $v_1$ and $v_2$ be the following paths respectively: 
	
	\begin{figure}[H]
		\centering
		\footnotesize{ $v_1:=13231 \rightarrow  31231 \rightarrow  31213 \rightarrow  32123 \rightarrow  31213 \rightarrow  13213 \rightarrow  13231 \rightarrow  12321 \rightarrow  13231$.}\label{pathp} \\ 
		\vspace{0.2cm}
		\footnotesize{ $v_2:=13231 \rightarrow  12321 \rightarrow  13231 \rightarrow  31231 \rightarrow  31213 \rightarrow   32123 \rightarrow  31213 \rightarrow  13213 \rightarrow  13231$.}\label{pathq}
	\end{figure}
	
	To simplify calculations we need the following. From Equation (\ref{eq1}) we obtain 
	
	\begin{equation}\label{eq3}
		f_{i(i+1)}(1{\otimes_{s_i}}x_{i+1}{\otimes_{s_{i+1}}}1{\otimes_{s_i}}1) = 1{\otimes_{s_{i+1}}}1{\otimes_{s_i}}1{\otimes_{s_{i+1}}}x_{i+2}.
	\end{equation}
	
	Similarly, from Equation (\ref{eq2}) we obtain 
	
	\begin{equation}\label{eq6}	
		f_{i(i-1)}(1{\otimes_{s_i}}1{\otimes_{s_{i-1}}}x_{i}{\otimes_{s_i}}1) = x_{i-1}{\otimes_{s_{i-1}}}1{\otimes_{s_i}}1{\otimes_{s_{i-1}}}1.
	\end{equation}
	
	We will use a diagrammatic method to evaluate homomorphisms
	in $\mathbb{S}Bim$. The next figure shows the evaluation of $f(v_1)$ (left) and $f(v_2)$ (right) in the element $x$ defined above. 
	
	\begin{figure}[H]
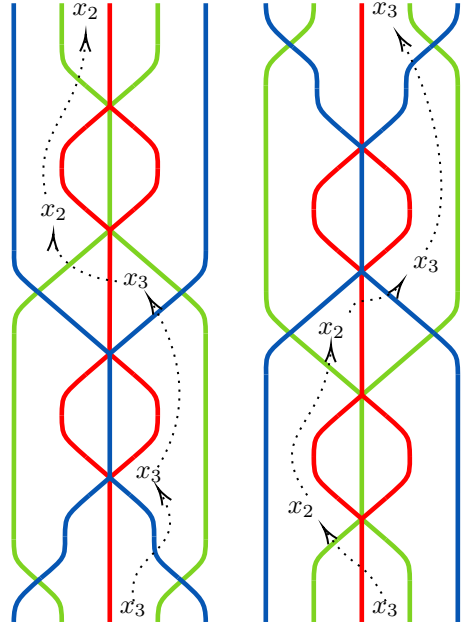

		\centering
		\tikzset{every picture/.style={line width=0.75pt}} 

		\caption{Diagrams for $v_1$ and $v_2$, evaluated in $x$.}\label{v1v2}
	\end{figure}

	We will prove that the elements obtained are different. It is known that 
	\begin{equation*}
		B_3B_2B_3 \cong B_{323} \oplus B_3 \text{ , where } B_{323} \cong \langle1^{\otimes} \rangle \text{ and }  B_3 \cong \langle 1\otimes 1\otimes x_3 \otimes 1 \rangle .
	\end{equation*}
	
	So it is possible to generate the $R$-bimodule $B_3B_2B_3$ with the elements $1^{\otimes}$ and $1\otimes 1\otimes x_3 \otimes 1$. In $B_1B_3B_2B_3B_1$, we have 
	\begin{equation*}
		f(v_1)(x) = 1 \otimes_{s_1} 1 \otimes_{s_3} 1 \otimes_{s_2} x_3 \otimes_{s_3} 1 \otimes_{s_1} 1
	\end{equation*}
	and
	\begin{equation*}
		f(v_2)(x) = 1 \otimes_{s_1} x_2 \otimes_{s_3} 1 \otimes_{s_2} 1 \otimes_{s_3} 1 \otimes_{s_1} 1.
	\end{equation*}
	
	If they were the same, applying dots over both $B_1$ in $B_1B_3B_2B_3B_1$ we would have that 
	$ x_2 \cdot 1^{\otimes} = 1\otimes_{s_3} 1 \otimes_{s_2} x_3 \otimes_{s_3} 1$ in $B_3B_2B_3$, which, as stated above, is not true.

	
	\section{A family of counterexamples}\label{generalization}
	
	The element $\sigma=12321$ is the only one where the FPC fails for the group $S_4$. We proved this by showing that the diagrams in Figure \ref{fig1} (the same diagrams as in Figure \ref{v1v2}) are not equal.
	
	\begin{figure}[H]
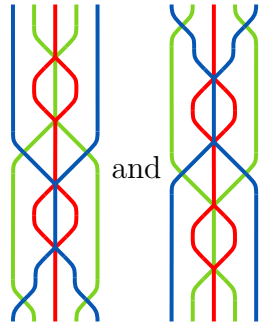

		\centering
		\tikzset{every picture/.style={line width=1pt}} 
		

		\caption{Diagrams for $f(v_1)$ and $f(v_2)$.}\label{fig1}
	\end{figure}
	
	The first diagram decomposes as in Figure \ref{fig3.14}, while the second decomposes as in Figure \ref{fig3.15}. The only summand that is different is the last one.
	\begin{figure}[H]
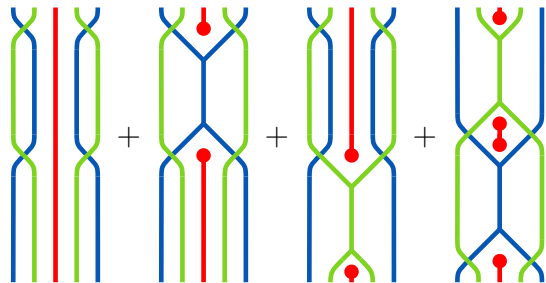

		\centering
		\tikzset{every picture/.style={line width=0.75pt}} 
		%
		\caption{Decomposition of $f(v_1)$.} 
		\label{fig3.14}
	\end{figure}
	
	\begin{figure}[H]
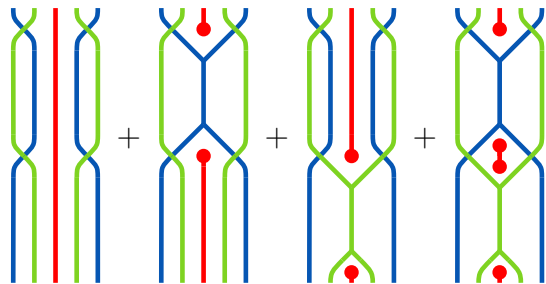

		\centering
		\tikzset{every picture/.style={line width=0.75pt}} 
		%
		\caption{Decomposition of $f(v_2)$.}
		\label{fig3.15}
	\end{figure}
	
	This allowed us to find the correct element to find the counterexample. Repeating the same idea, we see that the elements of the form $12\ldots (n-1)n(n-1)\ldots 21$ have a line as conflated expression graph. Let's say this line is the following figure.  
	
	\begin{figure}[H]
		\centering
		\tikzset{every picture/.style={line width=0.75pt}} 
		%
		\caption{}\label{fig20}
	\end{figure}
	
	Considering $p$ the path $[E_2, E_1, E_2, E_3, \ldots, E_n, E_{n-1}, \ldots, E_2]$ and $q$ the path $[E_2, E_3, \ldots, E_n, E_{n-1}, \ldots, E_1, E_2]$. We can check in general that $f(p)\neq f(q)$ by evaluating these path morphisms in particular elements. We will not give a rigorous proof of this fact, but the general strategy can be inferred from Figure \ref{counters}. The purple strand is related to the index $4$. The black, to the index $5$.
	
	\begin{figure}[H]
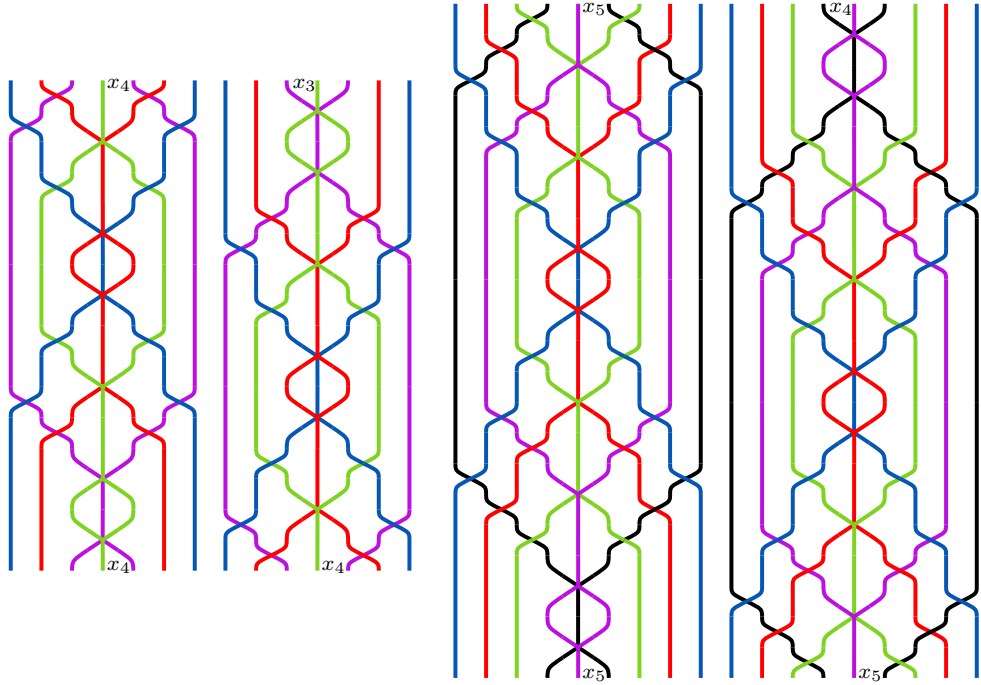

		\centering
		\tikzset{every picture/.style={line width=0.75pt}} 
		

		\caption{Induced morphisms for paths in $\Gamma_{1234321}$ and $\Gamma_{123454321}$.}\label{counters}
	\end{figure}
	
	\begin{remark}
		There are also elements of symmetric groups which can be used to produce counterexamples, and whose conflated expression graphs are not linear, i.e., elements whose conflated expression graph is different from Figure \ref{fig20}. The reader can verify that one such element is $12134325$ in $S_6$.
	\end{remark}
	
	Note that all elements in our family of counterexamples are paths with $p_a\notin \{\mathbf{s}, \mathbf{t}\}$ and $p_z\notin \{\mathbf{s}, \mathbf{t}\}$. This could make us think that the behavior for complete paths with $p_a\in \{\mathbf{s}, \mathbf{t}\}$ or $p_z \in \{\mathbf{s}, \mathbf{t}\}$ is different. That is not the case! Consider the same element $\sigma=12321$, and $\Gamma_{\sigma}$, the paths $[\mathbf{s},c,\mathbf{t},c,\mathbf{s},c]$ and $[\mathbf{s},c,\mathbf{t},c]$ give a counterexample. 
	It is enough to evaluate both path morphisms in the element $$1 \otimes_{s_1} x_2 \otimes_{s_2} 1 \otimes_{s_3} 1 \otimes_{s_2} 1 \otimes_{s_1} 1.$$
	
	These counterexamples (and some others that we do not show here) have in common particular choices of elements and paths, however, verification for other families of elements show that there is a phenomenon hidden underneath. To be precise, we have observed that for the longest element in $S_n$, any path $p$ with $p_a=\mathbf{s}$ and $p_z=\mathbf{t}$ will be equivalent to an oriented path from $\mathbf{s}$ to $\mathbf{t}$ (we proved this for $S_4$ in Section \ref{diagrammaticcalculus}). In other words, we do not need to follow the Manin-Schechtman orientation as long as we start and end in the right vertices. The same when we start in $\mathbf{t}$ and end in $\mathbf{s}$. We propose the following strengthening of the FPC for $w_0$.
	
	\begin{conjecture}
		Let $w_{0,n} \in S_n$ and $\mathbf{s}$, $\mathbf{t}$ be the source and sink of the Manin-Schechtman orientation. Let $p, q$ be two paths in $\Gamma_{w_{0,n}}$, which pass through $\mathbf{s}$ and $\mathbf{t}$, satisfying $p_a=q_a$, and $p_z=q_z$. Then $f(p)=f(q)$.
	\end{conjecture}
	
	We also conjecture the same for other choices of $\mathbf{s}$ and $\mathbf{t}$ obtained from other orientations different from the one given by Manin and Schechtman. Of course, this conjecture implies the FPC for $w_{0,n}$.
	
	\bibliographystyle{plain}
	\bibliography{references}

\begin{thebibliography}{1}

\bibitem{elias2016thicker}
B.~Elias.
\newblock {Thicker Soergel calculus in type A}.
\newblock {\em Proceedings of the London Mathematical Society},
  112(5):924--978, 2016.

\bibitem{elias2010diagrammatics}
B.~Elias and M.~Khovanov.
\newblock {Diagrammatics for Soergel Categories}.
\newblock {\em International Journal of Mathematics and Mathematical Sciences},
  2010:978635:1--978635:58, 2010.

\bibitem{Libedinsky2011new}
N.~Libedinsky.
\newblock {New bases of some Hecke algebras via Soergel bimodules}.
\newblock {\em Advances in Mathematics}, 228(2):1043--1067, 2011.

\bibitem{libGentle}
N.~Libedinsky.
\newblock {Gentle introduction to Soergel bimodules I: the basics}.
\newblock {\em S{\~a}o Paulo Journal of Mathematical Sciences}, pages 1--40,
  2017.

\bibitem{manin1989arrangements}
Y.~Manin and V.~Schechtman.
\newblock {Arrangements of hyperplanes, higher braid groups and higher Bruhat
  orders}.
\newblock {\em Algebraic number theory-Advanced Studies in Pure Mathemathics},
  17:289--308, 1989.

\bibitem{soergel2007kazhdan}
W.~Soergel.
\newblock {Kazhdan-Lusztig polynomials and indecomposable bimodules over
  polynomial rings}.
\newblock {\em Journal of the Institute of Mathematics of Jussieu},
  6(3):501--525, 2007.

\bibitem{williamson2011singular}
G.~Williamson.
\newblock {Singular Soergel bimodules}.
\newblock {\em International Mathematics Research Notices},
  2011(20):4555--4632, 2011.

\end{thebibliography}
	
\end{document}